\documentclass{amsart}
\usepackage{amssymb,amsmath,accents}
\usepackage{amscd}
\usepackage{amsfonts,amsthm,mathrsfs}
\usepackage{setspace} 
\usepackage{graphics}
\usepackage[dvips]{graphicx}
\usepackage{latexsym}
\usepackage{xcolor}

\usepackage{url}

\newtheorem{theorem}{Theorem}
\newtheorem{lemma}[theorem]{Lemma}
\newtheorem{proposition}[theorem]{Proposition}
\newtheorem{corollary}[theorem]{Corollary}

\theoremstyle{definition}
\newtheorem{definition}[theorem]{Definition}
\newtheorem{remark}[theorem]{Remark}

\newcommand{\id}{\,\mathrm{d}}

\providecommand{\keywords}%[1]
{
  \small	
  \textbf{\textit{Keywords:}}% #1
}

\numberwithin{equation}{section}

\begin{document}
%\vspace{0.7cm}

\title{On a singular limit of the Kobayashi--Warren--Carter energy}
\author[Y.~Giga]{Yoshikazu Giga}
\address[Y.~Giga]{Graduate School of Mathematical Sciences, The University of Tokyo,
3-8-1 Komaba, Meguro-ku, Tokyo 153-8914, Japan}
\email{labgiga@ms.u-tokyo.ac.jp}
\author[J.~Okamoto]{Jun Okamoto}
\address[J.~Okamoto]{Kyoto University Institute for the Advanced Study of Human Biology,
Yoshida-Konoe-cho, Sakyo-ku, Kyoto 606-8501, Japan}
\email{okamoto.jun.8n@kyoto-u.ac.jp}
\author[K.~Sakakibara]{Koya Sakakibara}
\address[K.~Sakakibara]{Department of Applied Mathematics, Faculty of Science, Okayama University of Science,
1-1 Ridaicho, Okayama-shi, Okayama 700-0005, Japan;
RIKEN iTHEMS,
2-1 Hirosawa, Wako-shi, Saitama 351-0198, Japan}
\email{ksakaki@ous.ac.jp}
\author[M.~Uesaka]{Masaaki Uesaka}
\address[M.~Uesaka]{Arithmer Inc.,
1-6-1 Roppongi, Minato-ku, Tokyo 106-6040, Japan;
Graduate School of Mathematical Sciences, The University of Tokyo,
3-8-1 Komaba, Meguro-ku, Tokyo 153-8914, Japan}
\email{masaaki.uesaka@arithmer.co.jp}
\begin{abstract}
By introducing a new topology, a representation formula of the Gamma limit of the Kobayashi--Warren--Carter energy is given in a multi-dimensional domain.
 A key step is to study the Gamma limit of a single-well Modica--Mortola functional.
 The convergence introduced here is called the sliced graph convergence, which is finer than conventional $L^1$ convergence, and the problem is reduced to a one-dimensional setting by a slicing argument.
\end{abstract}
\maketitle

\tableofcontents

%%%%%%%
\section{Introduction} \label{S1} % Section 1

We consider the Kobayashi--Warren--Carter energy, which is a sum of a weighted total variation and a single-well Modica--Mortola energy.
 Their explicit forms are
\begin{align}
	E^\varepsilon_\mathrm{KWC}(u,v) &:= \int_\Omega \alpha(v)|Du| 
	+ E^\varepsilon_\mathrm{sMM}(v), \label{eq:E_KWC}\\
	{E^\varepsilon_\mathrm{sMM}(v)} &:= \frac\varepsilon2 \int_\Omega |\nabla v|^2 {\mathrm{d}\mathcal{L}^N}
	+ \frac{1}{2\varepsilon} \int_\Omega F(v){\mathrm{d}\mathcal{L}^N},\notag
\end{align}
where $\Omega$ is a bounded domain in $\mathbf{R}^N$ with the Lebesgue measure $\mathcal{L}^N$, $\alpha\ge0$, $\varepsilon>0$ is a small parameter, and $F$ is a single-well potential which takes its minimum at $v=1$.
 Typical examples of $\alpha$ and $F$ are $\alpha(v)=v^2$ and $F(v)=(v-1)^2$, respectively.
 These are the original choices in \cite{KWC1,KWC3}.
% 原稿 2/12
 The first term in \eqref{eq:E_KWC} is a weighted total variation with weight $\alpha(v)$.
 This energy was first introduced by {\cite{KWC1,KWC3}} to model motion of {grain boundaries of polycrystal} which have some structures like the averaged angle of each grain.
 This energy is quite popular in materials science.

We are interested in a singular limit of the Kobayashi--Warren--Carter energy $E^\varepsilon_\mathrm{KWC}$ as $\varepsilon$ tends to zero.
 If we assume boundedness of $E^\varepsilon_\mathrm{KWC}$ for a sequence $(u,v_\varepsilon)$ for fixed $u$, {then} $v_\varepsilon$ tends to a unique minimum of $F$ as $\varepsilon\to0$ in the $L^2$ sense.
 However, if $u$ has a jump discontinuity, its convergence is not uniform near such places, even in a one-dimensional setting, suggesting that we have to introduce a finer topology than $L^2$ or $L^1$ topology.
 In fact, in a one-dimensional setting, the notion of graph convergence of $v^\varepsilon$ to a set-valued function is introduced, and representations of Gamma limits of $E^\varepsilon_\mathrm{KWC}$ and $E^\varepsilon_\mathrm{sMM}$ are given in \cite{GOU}.

% 原稿 3/12
In this paper, we extend this one-dimensional results to a multi-dimensional setting.
 For this purpose, we introduce a new concept of convergence called sliced graph convergence.
 Roughly speaking, it requires graph convergence on each line.
 Under this convergence in $v_\varepsilon$ and the $L^1$-convergence in $u$, one is able to derive a representation formula for the Gamma limit of $E^\varepsilon_\mathrm{KWC}$ as $\varepsilon\to0$.
 It is
\begin{align*}
	E^0_\mathrm{KWC}(u,\Xi) &:= \alpha(1) \int_{\Omega\backslash J_u} |Du| 
	+ \int_{J_u} \min_{\xi^-\leq\xi\leq\xi^+} \alpha(\xi) \left|u^+ - u^- \right| {\mathrm{d}\mathcal{H}^{N-1}} + {E^0_\mathrm{sMM}(\Xi)}, \\
	{E^0_\mathrm{sMM}(\Xi)} &:= 2 \int_\Sigma \left\{ G(\xi^-) + G(\xi^+) \right\} {\mathrm{d}\mathcal{H}^{N-1}}\notag
\end{align*}
when $v_\varepsilon$ converges to a set-valued function $\Xi$ of form
\begin{equation*}
	\Xi(z) = 
	\begin{cases}
		\left[\xi^-(z),\xi^+(z)\right],&z\in\Sigma,\\
		\{1\},&z\not\in\Sigma,
	\end{cases}
\end{equation*}
where $\Sigma$ is a countably $N-1$ rectifiable set, and $\xi^\pm$ are {$\mathcal{H}^{N-1}$-measurable functions with $\xi^- \le 1 \le \xi^+$. Here}  $\mathcal{H}^{N-1}$ denotes the $N-1$ dimensional Hausdorff measure.
 The function $G$ is defined by
\[
	G(\sigma) := \left| \int^\sigma_1 \sqrt{F(\tau)} {\mathrm{d}\tau} \right|.
\]
The functions $u^+$ and $u^-$ denote upper and lower approximate limits in the measure-theoretic sense \cite{Fe}.

% 原稿 4/12
In the case $\alpha(v)=v^2$, we see
\[
	E^0_\mathrm{KWC}(u,\Xi) = \int_{\Omega\backslash J_u} |Du| 
	+ \int_{J_u\cap\Sigma} \left(\xi^-_+ \right)^2 \left|u^+ - u^- \right| {\mathrm{d}\mathcal{H}^{N-1}}
	+ {E^0_\mathrm{sMM}(\Xi)},
\]
where $a_+$ denotes the positive part of a function $a$, i.e., $a_+=\max(a,0)$.
 In \cite{GOU}, the case $\alpha(v)=v^2$ is discussed for a one-dimensional setting.
 (Unfortunately, $\xi^-_+$ has been misprinted as $\xi^-$ in \cite{GOU}.)
 When $F(v)=(v-1)^2$,
\[
	{E^0_\mathrm{sMM}(\Xi)} = \int_\Sigma \left\{ (\xi^- -1)^2 + (\xi^+ -1)^2 \right\} {\mathrm{d}\mathcal{H}^{N-1}}.
\]
In a one-dimensional setting, the results in \cite{GOU} gave a full characterization of the Gamma limit: the compactness result and the mere convergence result. 
 On the other hand, it is unclear what kind of set-valued functions should be considered {as} the limit of $v_\varepsilon$ in a multi-dimensional setting, assuming $E^\varepsilon_\mathrm{sMM}(v_\varepsilon)$ is bounded.
 A compactness result is still missing in a multi-dimensional setting.

The basic idea is to reduce a multi-dimensional setting to a one-dimensional setting by a slicing argument based on the following disintegration
\[
	\int_\Omega f(z) {\mathrm{d}\mathcal{L}^N}(z)
	= \int_{{\pi_\nu(\Omega_\nu)}} \left(\int_{\pi^{-1}_\nu(x)} f\id\mathcal{H}^1\right)\id\mathcal{L}^{N-1}(x), 
\]
where $\pi_\nu$ denotes the projection of $\mathbf{R}^N$ to the {subspace} orthogonal to a unit vector $\nu$, and $\Omega_\nu=\pi_\nu(\Omega)$.
 This idea is often used to study the singular limit of the Ambrosio--Tortorelli functional
\[
	\mathcal{E}^\varepsilon (u,v) = \int_\Omega v^2 |\nabla u|^2\id\mathcal{L}^N
	+ \lambda \int_\Omega (u-h)^2\id\mathcal{L}^N + {E^\varepsilon_\mathrm{sMM}(v)}, \quad \lambda \geq 0
\] 
% 原稿 5/12
as in \cite{AT,AT2,FL}, where $h$ is a given $L^2$ function and $F(v)=(v-1)^2$.
 This problem can be handled in $L^1$ topology, and its limit is known to be the Mumford--Shah functional
\[
	\mathcal{E}^0 (u,K) = \int_{\Omega\backslash K} |\nabla u|^2\id\mathcal{L}^N
	+ \mathcal{H}^{N-1}(K) + \lambda \int_\Omega (u-h)^2\id\mathcal{L}^N,
\]
where $K$ is a countably $N-1$ rectifiable set.
 In this case, in our language, it suffices to consider the case $\xi^-=0$, $\xi^+=1$ on $\Sigma=K$ so that
\[
	{\mathcal{E}^0_\mathrm{sMM}(\Xi) }= \mathcal{H}^{N-1}(K).
\]
In our case, however, as observed in the one-dimensional problem \cite{GOU}, it is reasonable to study non-constant $\xi^\pm$.
 Moreover, the fidelity term including $\lambda$ is also allowed in our case.

Our first main result is the Gamma-convergence of
\[
	{E^\varepsilon_\mathrm{sMM} (v)} + \int_J \alpha(v) j(y) {\mathrm{d}\mathcal{H}^{N-1}}(y)
\]
for a given countably rectifiable set $J$, where $j$ is an $ \mathcal{H}^{N-1}${-}integrable function on $J$.
This energy is a special case of $E^\varepsilon_\mathrm{KWC} (u,v)$ when $u$ has a jump in $J$ while it is constant outside $J$.
 To show liminf inequality, we decompose $\Sigma$ into a disjoint union of compact sets {$\{K_i\}_i$} lying in almost flat hypersurfaces.
% 原稿 6/12
 Then we reduce the problem in a one-dimensional setting like \cite{FL}.
 To show limsup inequality, we approximate $\xi^\pm$ so that they are constants in each $K_i$.
 This approximation procedure is quite involved because one should approximate not only energies but also approximate in the sliced graph topology.
 The basic choice of recovery sequences is similar to \cite{AT,FL}.

This paper's main result is the Gamma-convergence of the Kobayashi--Warren--Carter energy.
 The additional difficulty comes from the $\int\alpha(v)|Du|$ part, and this part can be carried out by decomposing the domain of integration into two parts: place close to $\Sigma$ of the limit $\Xi$ of $v_\varepsilon$, and outside such place.

The most difficult problem is how to choose a suitable topology for $v_\varepsilon$ to $\Xi$.
 We take a slice, a straight line passing through $x$ with direction $\nu$ for $\mathcal{L}^{N-1}$-almost every $x\in\pi_\nu(\Omega)$ for some directions $\nu$. 
 We need several concepts of set-valued functions to formulate the topology, including measurability, as discussed in \cite{AF}.

The compactness is missing for the convergence of $E^\varepsilon_\mathrm{KWC}$ to $E^0_\mathrm{KWC}$.
Therefore, we do not know whether a minimizer of $E^0_\mathrm{KWC}$ exists under suitable boundary conditions or a minimizer of energy like $E^0_\mathrm{KWC}+\lambda\int_\Omega(u-h)^2 d\mathcal{L}^N$ exists.
 If one minimizes $E^0_\mathrm{KWC}$ in the $\Xi$ variable, i.e.,
\[
	TV_\mathrm{KWC}(u) := \inf_{\Xi\in\mathcal{A}_0} E^0_\mathrm{KWC}(u,\Xi),
\]
this can be calculated as
\[
	TV_\mathrm{KWC}(u) = \int_\Sigma \sigma \left( |u^+ - u^-| \right) {\mathrm{d}\mathcal{H}^{N-1}} + \int_{\Omega\backslash J_u}|Du|
\]
with
\begin{align*}
	\sigma(r) & := \min_{\xi^-,\xi^+} \left\{ r\min_{\xi^-\leq\xi\leq\xi^+} \alpha(\xi) + 2 \left( G(\xi^-)+G(\xi^+) \right) \right\} \\
	& = \min_{\xi^-} \left\{ r\min_{\xi^-\leq\xi\leq1} \alpha(\xi) + 2G(\xi^-) \right\},\ r \geq 0
\end{align*}
if $\alpha(v)\geq\alpha(1)$ for $v\geq1$.
 This $\sigma$ is always concave.
 If $F(v)=(v-1)^2$, then
\[
	\sigma(r) = \min_{\xi^-} \left\{ r(\xi^-_+)^2 + (\xi^- - 1)^2 \right\} = \frac{r}{r+1}.
\]
In other words,
\[
	TV_\mathrm{KWC}(u) = \int_\Sigma \frac{|u^+ - u^-|}{1+|u^+ - u^-|}{\mathrm{d}\mathcal{H}^{N-1}}
	+ \int_{\Omega\backslash J_u}|Du|.
\]
This functional is a kind of total variation but has different aspects.
 For example, if $u$ is a piecewise constant monotone increasing function in a one-dimensional setting, the total variation $TV(u)=\int_\Omega|Du|$ equals $\sup u-\inf u$.
 This case is often called a staircase problem since $TV$ does not care about the number and size of jumps for monotone functions.
 In contrast to $TV$, the $TV_\mathrm{KWC}$ costs less if the number of jumps is smaller, provided that each jump is the same size and $\sup u-\inf u$ is the same.
 The energy like $TV_\mathrm{KWC}$ for a piecewise constant function is derived as the surface tension of grain boundaries in polycrystals \cite{LL}, which is an active area, as studied by \cite{GaSp}.

The Modica--Mortola functional is the sum of Dirichlet energy and potential energy.
 The Gamma limit problem was first studied in \cite{MM1}.
 Since then, there has been much literature studying the Gamma-convergence problems.
% 原稿 7/12
 If $F$ is a double-well potential, say $F(v)=(v^2-1)^2$, then the Modica--Mortola functional reads
\[
	E^\varepsilon_\mathrm{dMM}(v) = \frac{\varepsilon}{2} \int_\Omega |\nabla v|^2 {\mathrm{d}\mathcal{L}^N}
	+ \frac{1}{2\varepsilon} \int_\Omega (v^2-1)^2 {\mathrm{d}\mathcal{L}^N}.
\]
If $E^\varepsilon_\mathrm{dMM}(v_\varepsilon)$ is bounded, $v_\varepsilon(z)$ converges to either $1$ or $-1$ for $\mathcal{L}^N$-almost all $z\in\Omega$ by taking a subsequence.
 The interface between two states, $\{\lim v_\varepsilon=1\}$ {and} $\{\lim v_\varepsilon=-1\}$, is called a transition interface.
 In a one-dimensional setting, its Gamma limit is considered in $L^1$ topology and is characterized by the number of transition points \cite{MM2}.
 This result is extended to a multi-dimensional setting in \cite{M,St}, and the Gamma limit is a constant multiple of the surface area of the transition interface.
 However, the topology of convergence of $v_\varepsilon$ is either {in} $L^1$ {topology or in measure} (including almost everywhere convergence).
 If we consider its Gamma limit in the sliced graph convergence, we expect that the limit equals
\[
	E^0_\mathrm{dMM}(\Xi) = 2 \int_\Sigma \left\{G_-(\xi^-)+G_+(\xi^+) \right\} {\mathrm{d}\mathcal{H}^{N-1}}
	+ G_-(1)\mathcal{H}^{N-1} (\Sigma)
\]
for
\begin{equation*}
	\Xi(z) := \left \{
	\begin{array}{ll}
	\left[ \xi^-(z), \xi^+(z) \right], &{\text{for}}\ z \in \Sigma, \\
	\text{either}\ 1\ \text{or}\ -1, &{\text{otherwise}},
	\end{array}
	\right.
\end{equation*}
where $[-1,1]\subset[\xi^-,\xi^+]$.
 Here, $G_\pm$ is defined as
\[
	G_\pm(\sigma) = \left| \int^\sigma_{\pm 1} \sqrt{F(\tau)}{\mathrm{d}\tau} \right|.
\]
The first term in $E_{\mathrm{dMM}}^0(\Xi)$ is invisible in $L^1$ convergence, while the second term is the Gamma limit of $E_\mathrm{dMM}$ in the $L^1$ sense.
 We do not give proof in this paper.
% 原稿 8/12
 If compactness is available, the Gamma-convergence yields the convergence of a local minimizer and the global minimizer.
 For $L^1$ convergence, based on this strategy, the convergence of a local minimizer has been established in \cite{KS} when the limit is a strict local minimizer.
 The convergence of critical points is outside the framework of a general theory and should be discussed separately as in \cite{HT}.
 In recent years, the Gamma limit of the double-well Modica--Mortola function with spatial inhomogeneity has been studied from a homogenization point of view (see, e.g.\ \cite{CFHP1}, \cite{CFHP2}) but still under $L^1$ or convergence in measure.

The Mumford--Shah functional $\mathcal{E}^0$ is difficult to handle because one of the variables is a set $K$.
 This is the motivation for introducing $\mathcal{E}^\varepsilon$, called the Ambrosio--Tortorelli functional, to approximate $\mathcal{E}^0$ in \cite{AT}.
 The Gamma limit of $\mathcal{E}^\varepsilon$ is by now well studied \cite{AT,AT2}, and with weights \cite{FL}.
 The convergence of critical parts is studied in \cite{FLS} in a one-dimensional setting; the higher-dimensional case was studied quite recently by \cite{BMR} by adjusting the idea of \cite{LSt}.
 The Ambrosio--Tortorelli approximation is now used in various problems, including the decomposition of brittle fractures \cite{FMa} and the Steiner problem \cite{LS,BLM}.
 However, in all these works, the energy for $u$ is a $v$-weighted Dirichlet energy, not $v$-weighted total variation energy.

% 原稿 9/12
A singular limit of the gradient flow of the double-well Modica--Mortola flow is well studied.
 The sharp interface limit, i.e., $\varepsilon\to0$ yields the mean curvature flow of an interface.
 For an early stage of development, see \cite{BL,XC,MSch}, on convergence to a smooth mean curvature flow and \cite{ESS} on convergence to a level-set mean curvature flow \cite{G}.
 For more recent studies, see, for example, \cite{AHM,To}.

The gradient flow of the Kobayashi--Warren--Carter energy $E^\varepsilon_\mathrm{KWC}$ is proposed in  \cite{KWC1} (see also  \cite{KWC2,KWC3}) to model grain boundary motion when each grain has some structure.
 Its explicit form is
\begin{align*}
	\tau_1 v_t &= s \Delta v + (1-v) - 2sv |\nabla v|, \\
	\tau_0 v^2 u_t &= s \operatorname{div} \left( v^2\frac{\nabla u}{|\nabla u|}\right),
\end{align*}
where $\tau_0$, $\tau_1$, and $s$ are positive parameters.
 This system is regarded as the gradient flow of $E^\varepsilon_\mathrm{KWC}$ with $F(v)=(v-1)^2$, $\varepsilon=1$, {and} $\alpha(v)=v^2$.
 Because of the presence of the singular term $\nabla u/|\nabla u|$, the meaning of the solution itself is non-trivial since, even if $v\equiv1$, the flow is the total variation flow, and a non-local quantity determines the speed \cite{KG}.
% 原稿 10/12
 At this moment, the well-posedness of its initial-value problem is an open question.
 If the second equation is replaced by
\[
	\tau_0 (v^2+{\delta}) u_t = s \operatorname{div} \left( (v^2+\delta') \nabla u/|\nabla u| + \mu\nabla u \right)
\]
with $\delta>0$, $\delta'\geq0$ and $\mu\geq0$ satisfying $\delta'+\mu>0$, the existence and large-time behavior of solutions are established in \cite{IKY,MoSh,MoShW1,SWat,SWY,WSh} under several homogeneous boundary conditions.
 However, its uniqueness is only proved in a one-dimensional setting under $\mu>0$ \cite[Theorem 2.2]{IKY}.
 These results can be extended to the cases of non-homogeneous boundary conditions.
 Under non-homogeneous Dirichlet boundary conditions, we are able to find various structural patterns of steady states; see \cite{MoShW2}.

% 原稿 11/12
The singular limit of the gradient flow of $E^\varepsilon_\mathrm{KWC}$ is not known even if $\alpha(v)=v^2+\delta'$, $\delta'>0$.
 In \cite{ELM}, a gradient flow of
\[
	E(u,\Sigma) = \int_\Sigma \sigma \left(\left|u^+-u^-\right|\right){\mathrm{d}\mathcal{H}^{N-1}}, \quad N=2
\]
is studied.
 Here $u$ is a piecewise constant function outside a union $\Sigma$ of smooth curves, including triple junction, and $\sigma$ is a given non-negative function.
 %If we minimize $E^0_\mathrm{KWC}(u,\Xi)$ with $\alpha(v)^2$, $F(v)=(v-1)^2$ for fixed piecewise constant $u$ whose jump discontinuities in $\Sigma$ by taking optimal $\xi^-$ and $\xi^+$ (fixing $\Sigma$), the minimum value equals
%\[
%	E^0_m(u,\Sigma) = \int_\Sigma \frac{\left| u^+ - u^- \right|}{1+\left| u^+ - u^- \right|}{\id\mathcal{H}^{N-1}},
%\]
%which corresponds to the case $\sigma(r)=r/(1+r)$ of $E(u,\Sigma)$; this is easy since
%\[
%	\min_{x\in\mathbf{R}} px^2_+ + (x-1)^2 = \frac{p}{1+p}.
%\]
 Our $TV_\mathrm{KWC}$ is a typical example.
 They take variation of $E$ not only $u$ but also of $\Sigma$ and derive a weighted curvature flow with evolutions of boundary values of $u$ together with motion of triple junction.
 It is not clear that the singular limit of the gradient flow of $E^\varepsilon_\mathrm{KWC}$ gives this flow since, in the total variation flow, the variation is taken only in the direction of $u$ and does not include domain variation, which is the source of the mean curvature flow.

% 原稿 11/12
This paper is organized as follows.
 In Section \ref{SSGC}, we introduce the notion of sliced graph convergence.
 In Section \ref{SLSC}, we discuss the liminf inequality of the singular limit of $E^\varepsilon_\mathrm{sMM}$ with an additional term under the sliced graph convergence.
 In Section \ref{SCRS}, we discuss the limsup inequality by constructing recovery sequences.
 In Section \ref{SLKWC}, we discuss the singular limit of $E^\varepsilon_\mathrm{KWC}$.

{The results of this paper are based on the thesis \cite{O} of the second author.}
%%%%%%%
% 原稿 1/5
\section{Sliced graph convergence} \label{SSGC} % Section 2

In this section, we introduce the notion of sliced graph convergence.
We first recall a few basic notions of a set-valued function, especially on the measurability.
Consequently, we review the notion of the slicing argument and introduce the concept of sliced graph convergence.
\subsection{A set-valued function and its measurability}
We first recall a few basic notions of a set-valued function; see \cite{AF}.
 Let $M$ be a Borel set in $\mathbf{R}^d$ and $\Gamma$ be a set-valued function on $M$ with values in $2^{\mathbf{R}^m}\backslash\{\emptyset\}$ such that $\Gamma(z)$ is closed in $\mathbf{R}^m$ for all $z\in M$.
 We say that such $\Gamma$ is a closed set-valued function.
 We say that $\Gamma$ is \emph{Borel measurable} if $\Gamma^{-1}(U)$ is a Borel set whenever $U$ is an open set in $\mathbf{R}^m$.
 Here, the inverse $\Gamma^{-1}(U)$ is defined as
\[
	\Gamma^{-1}(U) := \left\{ z \in M \bigm|
	\Gamma(z) \cap U \neq \emptyset \right\}.
\]
Similarly, we say that $\Gamma$ is \emph{Lebesgue measurable} if $\Gamma^{-1}(U)$ is Lebesgue measurable whenever $U$ is an open set.

 Assume that $M$ is closed.
 We say that $\Gamma$ is \emph{upper semicontinuous} if $\operatorname{graph}\Gamma$ is closed in $M\times\mathbf{R}^m$, where
\[
	\operatorname{graph}\Gamma := \left\{ z=(x,y) \in M \times \mathbf{R}^m \bigm|
	y \in \Gamma({x}),\ x \in M \right\}.
\]
If $\Gamma$ is upper semicontinuous, $\Gamma$ is Borel measurable \cite{AF}.

 Assume that $M$ is compact.
 Then, $\operatorname{graph}\Gamma$ is compact if it is closed.
 We set
\[
	\mathcal{C} = \left\{ \Gamma \mid
	\operatorname{graph}\Gamma\ \text{is compact in}\ M\times \mathbf{R}^m\ \text{and}\ \Gamma(x)\neq\emptyset \ \text{for}\ {x} \in M\right\}.
\]
For $\Gamma_1, \Gamma_2 \in \mathcal{C}$, we set
\[
	d_g(\Gamma_1, \Gamma_2) := d_H(\operatorname{graph}\Gamma_1, \operatorname{graph}\Gamma_2),
\]
where $d_H$ denotes the Hausdorff distance of two sets in $M\times\mathbf{R}^m$, defined by
\[
	d_H(A,B) := \max \left\{ \sup_{x \in A} \operatorname{dist}(z,B), \sup_{w \in B} \operatorname{dist}(w,A) \right\}
\]
for $A,B \subset M\times\mathbf{R}^m$, and 
\[
	\operatorname{dist}(z,B) := \inf_{w \in B} \operatorname{dist}(z,w),\quad
	\operatorname{dist}(z,w) = |z-w|,
\]
where $|\cdot|$ denotes the Euclidean norm in $\mathbf{R}^d\times\mathbf{R}^m$.

We recall a fundamental property of a Borel measurable set-valued function \cite[Theorem 8.1.4]{AF}.
\begin{theorem} \label{MBA}
Let $\Gamma$ be a closed set-valued function on a Borel set $M$ in $\mathbf{R}^d$ with values in $2^{\mathbf{R}^m}\backslash\{\emptyset\}$.
 The following three statements are equivalent:
\begin{enumerate}
\item[(i)] $\Gamma$ is Borel (resp.\ Lebesgue) measurable.
\item[(i\hspace{-1pt}i)] $\operatorname{graph}\Gamma$ is a Borel set ($\mathbf{M}\otimes\mathbf{B}$ measurable set) in $M\times\mathbf{R}^m$.
\item[(i\hspace{-1pt}i\hspace{-1pt}i)] There is a sequence of Borel (Lebesgue) measurable functions $\{f_j\}^\infty_{j=1}$ such that
\[
	\Gamma(z) = \overline{\left\{f_j(z)\bigm| j=1,2,\ldots\right\}}.
\]
\end{enumerate}
Here $\mathbf{M}$ denotes the $\sigma$-algebra of Lebesgue measurable sets in $M$ and $\mathbf{B}$ denotes the $\sigma$-algebra of Borel sets in $\mathbf{R}^m$.
\end{theorem}
\subsection{The definition of the sliced graph convergence}
We next recall the notation often used in the slicing argument \cite{FL}.
 Let $S$ be a set in $\mathbf{R}^N$.
 Let $S^{N-1}$ denote the unit sphere in $\mathbf{R}^N$ centered at the origin, i.e.,
\[
	S^{N-1} = \left\{ \nu \in \mathbf{R}^N \bigm|
	|\nu| = 1 \right\}.
\]
% 原稿 2/5
For a given $\nu$, let $\Pi_\nu$ denote the hyperplane whose normal equals $\nu$.
 In other words,
\[
	\Pi_\nu := \left\{ x \in \mathbf{R}^N \bigm|
	\langle x,\nu \rangle = 0 \right\},
\]
where $\langle \ ,\ \rangle$ denotes the standard inner product in $\mathbf{R}^N$.
 For $x\in\Pi_\nu$, let $S_{x,\nu}$ denote the intersection of {$S$} and the whole line with direction $\nu$, which contains $x$; that is,
\[
	S_{x,\nu} := \left\{ x + t \nu \bigm|
	t \in S^1_{x,\nu} \right\},
\]
where
\[
	S^1_{x,\nu} := \left\{ t \in \mathbf{R} \bigm|
	x + t\nu \in S \right\} \subset \mathbf{R}.
\]
We also set
\[
	S_\nu := \left\{ x 	\in \Pi_\nu \bigm|
	S_{x,\nu} \neq \emptyset \right\}.
\]
See Figure \ref{FSC}.
\begin{figure}[htb]
\centering 
\includegraphics[width=5cm]{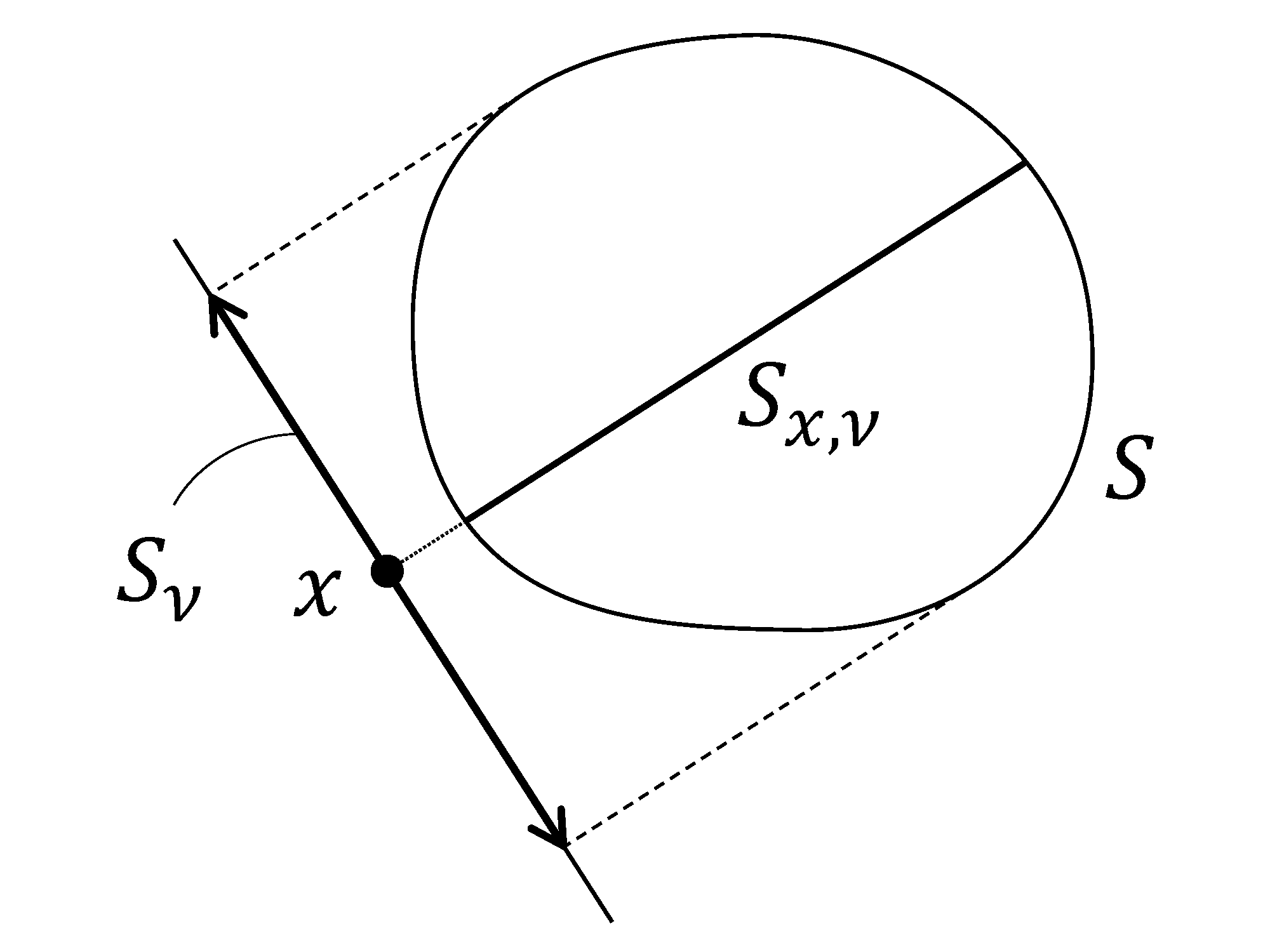}
\caption{Slicing} \label{FSC} % Slicing？
\end{figure}
 For a given function $f$ on $S$, we associate it with a function $f_{x,\nu}$ on $S^1_{x,\nu}$ defined by
\[
	f_{x,\nu}(t) := f(x + t \nu).
\]

%
% 原稿 4/5
Let $\Omega$ be a bounded domain in $\mathbf{R}^N$, and $\mathcal{T}$ denote the set of all Lebesgue measurable (closed) set-valued function $\Gamma:\Omega\to2^\mathbf{R}$.
For $\nu \in S^{N-1}$, we consider $\Omega^1_{x,\nu}\subset\mathbf{R}$ and the (sliced) set-valued function $\Gamma_{x,\nu}$ on $\Omega^1_{x,\nu}$ defined by $\Gamma_{x,\nu}(t)=\Gamma(x+t\nu)$.
 Let $\overline{\Gamma_{x,\nu}}$ denote its closure defined on the closure of $\overline{\Omega^1_{x,\nu}}$.
 Namely, it is uniquely determined so that the graph of $\overline{\Gamma_{x,\nu}}$ equals the closure of $\operatorname{graph}\Gamma_{x,\nu}$ in $\mathbf{R}\times\mathbf{R}$.
 As with usual measurable functions, $\Gamma^{(1)}$ and $\Gamma^{(2)}$ belonging to $\mathcal{T}$ are identified if $\Gamma^{(1)}(z)=\Gamma^{(2)}(z)$ for $\mathcal{L}^N$-a.e.\ $z\in\Omega$.
 By Fubini's theorem,  $\Gamma^{(1)}_{x,\nu}(t)=\Gamma^{(2)}_{x,\nu}(t)$ for $\mathcal{L}^1$-a.e.\ $t$ for $\mathcal{L}^{N-1}$-a.e.\ $x\in\Omega_\nu$.
 With this identification, we consider its equivalence class, and we call each $\Gamma^{(1)}$, $\Gamma^{(2)}$ a representative of this equivalence class.
 For $\nu\in S^{N-1}$, we define the subset $\mathcal{B}_\nu \subset \mathcal{T}$ as follows: $\Gamma \in \mathcal{B}_\nu$ if, for a.e.\ $x\in\Omega_\nu$,
 \begin{itemize}
     \item There is a representative of $\Gamma_{x,\nu}$ such that $\overline{\Gamma_{x,\nu}} = \Gamma_{x,\nu}$ on $\Omega^1_{x,\nu}$;
     \item $\operatorname{graph}\overline{\Gamma_{x,\nu}}$ is compact in $\overline{\Omega^1_{x,\nu}}\times\mathbf{R}$.
 \end{itemize}

We note that if $\Gamma^{(1)},\Gamma^{(2)}\in\mathcal{B}_\nu$, then $\overline{\Gamma^{(1)}_{x,\nu}},\overline{\Gamma^{(2)}_{x,\nu}}\in\mathcal{C}$ with $M=\overline{\Omega^1_{x,\nu}}$ by a suitable choice of representative of $\Gamma^{(1)}_{x,\nu}, \Gamma^{(2)}_{x,\nu}$, which follows from the definition.

In this situation, we have the following fact:
\begin{lemma}
The function 
\[
	f(x) = d_g \left( \overline{\Gamma^{(1)}_{x,\nu}},\overline{\Gamma^{(2)}_{x,\nu}} \right)
	= d_H \left( \operatorname{graph}\Gamma^{(1)}_{x,\nu},\operatorname{graph}\Gamma^{(2)}_{x,\nu} \right)
\]
is Lebesgue measurable in $\Omega_\nu$.
\end{lemma}\label{lemma:distance}
\begin{proof}
Since each Lebesgue measurable function $f$ has a Borel measurable function $\overline{f}$ with $f(z)=\overline{f}(z)$ for $\mathcal{L}^N$-a.e.\ $z\in\Omega$, by Theorem~\ref{MBA}~(i\hspace{-1pt}i\hspace{-1pt}i), there is a Borel measurable representative of $\Gamma$.
By Theorem~\ref{MBA}~(i\hspace{-1pt}i),
$\operatorname{graph}\Gamma$ is a Borel set for the Borel representative of $\Gamma$.
Since the graph of the set-valued function $T:x\longmapsto\operatorname{graph}\overline{\Gamma_{x,\nu}}$ on $\Omega_\nu$ equals $\operatorname{graph}\Gamma$ for $\Gamma\in\mathcal{B}_\nu$ by taking a suitable representative of $\Gamma$,
we see that $T$ should be Borel measurable if $\Gamma$ is Borel measurable by Theorem~\ref{MBA}~(i\hspace{-1pt}i).
 (Note that $T(x)$ is a compact set in $\mathbf{R}\times\mathbf{R}$.)
 Since $d_H$ is continuous, the map $f(x)$ should be measurable.
\end{proof}
% 原稿 5/5
We now introduce a metric on $\mathcal{B}_\nu$ of form 
\[
	d_\nu \left( \Gamma^{(1)},\Gamma^{(2)} \right)
	:= \int_{\Omega_\nu} \frac{d_g \left( \overline{\Gamma^{(1)}_{x,\nu}},\overline{\Gamma^{(2)}_{x,\nu}} \right)}{1+d_g \left( \overline{\Gamma^{(1)}_{x,\nu}},\overline{\Gamma^{(2)}_{x,\nu}} \right)} \id\mathcal{L}^{N-1}(x)
\]
for $ \Gamma^1,\Gamma^2\in\mathcal{B}_\nu$, where $\mathcal{L}^{N-1}$ denotes the Lebesgue measure on $\Pi_\nu$.
 From Lemma~\ref{lemma:distance}, we see that this is a well-defined quantity for all $ \Gamma^{(1)},\Gamma^{(2)}\in\mathcal{B}_\nu$.
 We identify $\Gamma^{(1)},\Gamma^{(2)}\in\mathcal{B}_\nu$ if $\Gamma^{(1)}_{x,\nu}=\Gamma^{(2)}_{x,\nu}$ for a.e.\ $x$.
 With this identification, $(\mathcal{B}_\nu,d_\nu)$ is indeed a metric space.
 By a standard argument, we see that $(\mathcal{B}_\nu,d_\nu)$ is a complete metric space; we do not give proof since we do not use this fact.

Let $D$ be a countable dense set in $S^{N-1}$.
 We set
\[
	\mathcal{B}_D := \bigcap_{\nu\in D}{\mathcal{B}_\nu}.
\]
It is a metric space with metric
\[
	d_D \left( \Gamma^{(1)},\Gamma^{(2)} \right)
	:= \sum^\infty_{j=1} \frac{1}{2^j}
	\frac{d_{\nu_j} \left( \Gamma^{(1)},\Gamma^{(2)} \right)}{1+d_{\nu_j} \left( \Gamma^{(1)},\Gamma^{(2)} \right)},
\]
where $D=\{\nu_j\}^\infty_{j=1}$.
 (This is also a complete metric space.)

We shall fix $D$.
The convergence with respect to $d_D$ is called the \emph{sliced graph convergence}.
If $\{\Gamma_k\}\subset\mathcal{B}_D$ converges to $\Gamma\in\mathcal{B}_D$ with respect to $d_D$, we write $\Gamma_k\xrightarrow{sg}\Gamma$ (as $k\to\infty$).
Roughly speaking, $\Gamma_k\xrightarrow{sg}\Gamma$ if the graph of the slice $\Gamma_k$ converges to that of $\Gamma$ for a.e. $x \in \Omega_\nu$ for any $\nu \in D$.
% 原稿 1/8
For a function $v$ on $\Omega$, we associate a set-valued function $\Gamma_v$ by $\Gamma_v(x)=\left\{v(x)\right\}$.
If $\Gamma_k=\Gamma_{v_k}$ for some $v_k$, we shortly write $v_k\xrightarrow{sg}\Gamma$ instead of $\Gamma_{v_k}\xrightarrow{sg}\Gamma$.
We note that if $v\in H^1(\Omega)$, the $L^2$-Sobolev space of order $1$, then $\Gamma_v\in\mathcal{B}_D$ for any $D$.

We conclude this subsection by showing that the notions of the graph convergence and the sliced graph convergence are unrelated for $N\geq2$.
 First, we give an example that the graph convergence does not imply the sliced graph convergence.
 Let $C(r)$ denote the circle of radius $r>0$ centered at the origin in $\mathbf{R}^2$.
 It is clear that $d_H\left(C(r),C(r-\varepsilon)\right)\to0$ as $\varepsilon>0$ tends to zero.
 However, for $\nu=(1,0)$, $C(r-\varepsilon)_{x,\nu}$ with $x=(0,\pm r)$ is empty and does not converge to a single point $C(r)_{x,\nu}=\left\{(0,\pm r)\right\}$.
 In this case, $C(r-\varepsilon)_{x,\nu}$ converges to $C(r)_{x,\nu}$ in the Hausdorff sense except the case $x=(0,\pm r)$.
 To make the exceptional set has a positive $\mathcal{L}^1$ measure in $\Pi_\nu$, we recall a thick Cantor set defined by
\begin{align*}
	G &:= [0,1] \backslash U \\
	U &:= \bigcup \left\{\left( \frac{a}{2^n} - \frac{1}{2^{2n+1}}, \frac{a}{2^n} + \frac{1}{2^{2n+1}} \right)
	\biggm| n, a = 1,2,\ldots	\right\}.
\end{align*}
This $G$ is a compact set with a positive $\mathcal{L}^1$ measure.
 We set
\[
	K := \bigcup_{r \in G} C(r), \quad
	K_\varepsilon := \bigcup_{r \in G} C(r-\varepsilon).
\]
$K_\varepsilon$ converges to $K$ as $\varepsilon\to0$ in the Hausdorff distance sense.
 However, for any $\nu\in S^2$, the slice $(K_\varepsilon)_{x,\nu}$ does not converge to {$K_{x,\nu}$} for $x\in\Pi_\nu$ with $|x|\in G$.
 It is easy to construct an example that the graph convergence does not imply the sliced graph convergence based on this set.
 Let $\Omega$ be an open unit disk centered at the origin.
 We set
\begin{center}
\begin{minipage}[c][24pt][b]{0.35\textwidth}
\begin{eqnarray*}
	\Gamma_\varepsilon(x) := \left\{
\begin{array}{cl}
	[0,1], & z \in K_\varepsilon \\
	\{ 1 \}, & z \in \Omega\backslash K_\varepsilon
\end{array}
\right.,
\end{eqnarray*}
\end{minipage}
\begin{minipage}[c][24pt][b]{0.35\textwidth}
\begin{eqnarray*}
	\Gamma(x) := \left\{
\begin{array}{cl}
	[0,1], & z \in K \\
	\{ 1 \}, & z \in \Omega\backslash K
\end{array}
\right..
\end{eqnarray*}
\end{minipage}
\end{center}
The graph convergence of $\Gamma_\varepsilon$ to $\Gamma$ is equivalent to the Hausdorff convergence of $K_\varepsilon$ to $K$.
 The sliced graph convergence is equivalent to saying $(K_\varepsilon)_{x,\nu}\to K_{x,\nu}$ for $\nu\in D$ and a.e.\ $x$, where $D$ is some dense set in $S^1$.
 However, from the construction of $K_\varepsilon$ and $K$, we observe that for any $\nu\in S^1$, the slice $K_{x,\nu}$ does not converge to $K$ for $x$ with $|x|\in G$, which has a positive $\mathcal{L}^1$ measure on $\Pi_\nu$.
 Thus, we see that $\Gamma_\varepsilon$ does not converge to $\Gamma$ in the sense of the sliced graph convergence while $\Gamma_\varepsilon$ converges to $\Gamma$ in the sense of graph convergence. % converge？

The sliced graph convergence does not imply the graph convergence even if the graph convergence is interpreted in the sense of essential distance.
 For any $\mathcal{H}^N$-measurable set $A$ in $\mathbf{R}^{N+1}$ and a point $p\in\mathbf{R}^{N+1}$, we set the essential distance from $p$ to $A$ as
\[
	d_e(p,A) := \inf \left\{ r>0 \bigm| \mathcal{H}^N \left( B_r(p)\cap A \right) > 0 \right\},
\]
where $B_r(p)$ is a closed ball of radius $r$ centered at $p$.
 We set
\[
	N_\delta(A) := \left\{ q\in\mathbf{R}^{N+1} \bigm| d_e (q,A) < \delta \right\},
\]
and the essential Hausdorff distance is defined as
\[
	d_{eH}(A,B) := \inf \left\{ \delta>0 \bigm| A \subset N_\delta(B),\ B \subset N_\delta(A) \right\}.
\]
Let $\Omega$ be a domain in $\mathbf{R}^N$ ($N\geq 2$) containing $B_1(0)$ and set
\[
	\Gamma^\varepsilon(z) = \left\{ \left( 1-|z|/\varepsilon \right)_+ \right\}, \quad
	\Gamma^0(z) = \{0\}
\]
for $z\in\Omega$ and $\varepsilon>0$.
 Clearly, for any $\nu\in S^{N-1}$, $x\in\Omega_\nu$ with $x\neq0$,
\[
	d_H \left( \operatorname{graph} \Gamma^\varepsilon_{x,\nu}, \operatorname{graph} \Gamma^0_{x,\nu} \right) \to 0
\]
holds as $\varepsilon\to0$,
 However,
\[
	d_{eH} \left( \operatorname{graph} \Gamma^\varepsilon, \operatorname{graph} \Gamma^0 \right) = 1;
\]
in particular, $\Gamma^\varepsilon$ does not converge to $\Gamma^0$ in the $d_{eH}$ convergence of the graphs.

%%%%%%%
\section{Lower semicontinuity} \label{SLSC} % Section 3

We now introduce a single-well Modica--Mortola function $E^\varepsilon_\mathrm{sMM}$ on $H^1(\Omega)$ when $\Omega$ is a bounded domain in $\mathbf{R}^N$.
 For $v\in H^1(\Omega)$, we set an integral
\[
	E^\varepsilon_\mathrm{sMM} (v) := \frac{\varepsilon}{2} \int_\Omega |\nabla v|^2 \id\mathcal{L}^N
	 + \frac{1}{2\varepsilon} \int_\Omega F(v) \id\mathcal{L}^N,
\]
where $\mathcal{L}^N$ denotes the $N$-dimensional Lebesgue measure.
 Here, the potential energy $F$ is a single-well potential.
 We shall assume that 
\begin{enumerate}
\item[(F1)] $F\in C^1(\mathbf{R})$ is non-negative, and $F(v)=0$ if and only if $v=1$, 
\item[(F2)] $\liminf_{|v|\to\infty} F(v) > 0$.
 We occasionally impose a stronger growth assumption than (F2): 
\item[(F2')] (monotonicity condition) $F'(v)(v-1)\geq0$ for all $v\in\mathbf{R}$.
\end{enumerate}

We are interested in the Gamma limit of $E^\varepsilon_\mathrm{sMM}$ as $\varepsilon\to 0$ under the sliced graph convergence.
We define the subset $\mathcal{A}_0 := \mathcal{A}_0(\Omega) \subset \mathcal{B}_D$ as follows: $\Xi \in \mathcal{A}_0(\Omega)$ if there is a countably $N-1$ rectifiable set $\Sigma\subset\Omega$ such that 
\begin{equation}
\Xi(z) = \left \{
\begin{array}{l} \label{SIG}
1,\ z\in\Omega\backslash\Sigma \\
\left[\xi^-, \xi^+\right],\ z\in\Sigma
\end{array}
\right.
\end{equation}
with $\mathcal{H}^{N-1}$-measurable function $\xi_\pm$ on $\Sigma$ and $\xi^-(z) \leq 1 \leq \xi^+(z)$ for $\mathcal{H}^{N-1}$-a.e.\ $z \in \Sigma$.
For the definition of countably $N-1$ rectifiability, see the beginning of Section~\ref{SSBP}.
% 原稿 2/8
Here $\mathcal{H}^m$ denotes the $m$-dimensional Hausdorff measure.

We briefly remark on the compactness of the graph of $\Xi\in\mathcal{A}_0$.
By definition, if $\Xi$ is of form \eqref{SIG}, then $\Xi(z)$ is compact.
However, there may be a chance that $\operatorname{graph}\overline{\Gamma_{x,\nu}}$ is not compact, even for the one-dimensional case ($N=1$).
 Indeed, if a set-valued function on $(0,1)$ is of form
\begin{equation*}
	\Xi(z) = \left \{
\begin{array}{ll}
\left[1,m\right]&\text{for}\ z=1/m \\
\{1\}&\text{otherwise},
\end{array}
\right.
\end{equation*}
then $\overline{\Xi}$ is not compact in $[0,1]\times\mathbf{R}$.
It is also possible to construct an example that $\overline{\Xi}\neq\Xi$ in $(0,1)$, which is why we impose $\Xi\in\mathcal{B}_D$ in the definition of $\mathcal{A}_0$.

For $\Xi\in\mathcal{A}_0$, we define a functional
\[
	E^0_\mathrm{sMM}(\Xi,\Omega) := 2\int_\Sigma \left\{ G(\xi^-) + G(\xi^+) \right\} {\id\mathcal{H}^{N-1}},\quad \text{where}\ G(\sigma) := \left| \int^\sigma_1 \sqrt{F(\tau)} {\id\tau} \right|.
\]
For later applications, it is convenient to consider a more general functional.
 Let $J$ be a countably $N-1$ rectifiable set, and $\alpha:\mathbf{R}\to[0,\infty)$ be continuous.
 Let $j$ be a non-negative $\mathcal{H}^{N-1}$-measurable function on $J$. We denote the triplet $(J,j,\alpha)$ by $\mathcal{J}$.
 We set
\[
	E^{0,{\mathcal{J}}}_\mathrm{sMM}(\Xi,\Omega) = E^0_\mathrm{sMM}(\Xi,\Omega)
	+ \int_{J\cap\Sigma} \left( \min_{\xi^-\leq\xi\leq\xi^+} \alpha(\xi) \right) {\id\mathcal{H}^{N-1}}.
\]
For $S$, we also set
\[
	E^{\varepsilon,{\mathcal{J}}}_\mathrm{sMM}(v) := E^\varepsilon_\mathrm{sMM}(v)
	+ \int_J \alpha(v)j{\id\mathcal{H}^{N-1}},
\]
which is important to study the Kobayashi--Warren--Carter energy.

%%%%%%%
\subsection{Liminf inequality} \label{SSLINF} % Subsection 3.1
We shall state the ``liminf inequality'' for the convergence of $E^{\varepsilon,{\mathcal{J}}}_\mathrm{sMM}$.
% 原稿 3/8
%
\begin{theorem} \label{INF}
Let $\Omega$ be a bounded domain in $\mathbf{R}^N$.
 Assume that $F$ satisfies (F1) and (F2).
 For ${\mathcal{J}}=(J,j,\alpha)$, assume that $J$ is countably $N-1$ rectifiable in $\Omega$ with a non-negative $\mathcal{H}^{N-1}$-measurable function $j$ on $J$ and that $\alpha \in C(\mathbf{R})$ is non-negative.
 Let $D$ be a countable dense set of $S^{N-1}$.
 Let $\{v_\varepsilon\}_{0<\varepsilon<1}$ be in $H^1(\Omega)$ so that $\Gamma_{v_\varepsilon}\in\mathcal{B}_D$.
 If $v_\varepsilon\xrightarrow{sg}\Xi$ and $\Xi\in\mathcal{A}_0$, then
\[
	E^{0,{\mathcal{J}}}_\mathrm{sMM}(\Xi,\Omega) \leq \liminf_{\varepsilon\to 0}E^{\varepsilon,{\mathcal{J}}}_\mathrm{sMM} (v_\varepsilon).
\]
\end{theorem}
\begin{remark} \label{INF1}
\begin{enumerate}
\item[(i)] The last inequality is called the liminf inequality.
 Here, we assume that the limit $\Xi$ is in $\mathcal{A}_0$, which is a stronger assumption than the one-dimensional result \cite[Theorem 2.1 (i)]{GOU}, where this condition automatically follows from the finiteness of the right-hand side of the liminf inequality.
\item[(i\hspace{-1pt}i)] In a one-dimensional setting, we consider the limit functional in $\overline{\Omega}$.
 Here we only consider it in $\Omega$.
 Thus, our definition of $\mathcal{A}_0$ is different from \cite{GOU}.
 Under suitable assumptions on the boundary, say $C^1$, we are able to extend the result onto $\overline{\Omega}$.
 Of course, we may replace $\Omega$ with a flat torus $\mathbf{T}^N=\mathbf{R}^N/\mathbf{Z}^N$.
\item[(i\hspace{-1pt}i\hspace{-1pt}i)] In \cite{GOU}, $\alpha(v)$ is taken $v^2$ so that
\[
	E^{0,b}_\mathrm{sMM}(\Xi,M) = E^0_\mathrm{sMM}(\Xi,M) + b\left(\left(\min\Xi(a)\right)_+\right)^2,
\]
where $(f)_+$ denotes the positive part defined by $f_+=\max(f,0)$.
 However, in \cite{GOU}, this operation was missing in the definition, which is incorrect.
\end{enumerate}
\end{remark}
%

%%%%%%%
\subsection{Basic properties of a countably $N-1$ rectifiable set} \label{SSBP} % Subsection 3.2

To prove Theorem \ref{INF}, we begin with the basic properties of a countably $N-1$ rectifiable set.
A set $J$ in $\mathbf{R}^N$ is said 
% 原稿 4/8
to be countably $N-1$ rectifiable if
\[
	J \subset J_0 \cup \left( \bigcup^\infty_{j=1} F_j\left(\mathbf{R}^{N-1}\right) \right)
\]
where $\mathcal{H}^{N-1}(J_0)=0$ and $F_j:\mathbf{R}^{N-1}\to\mathbf{R}^N$ are Lipschitz mappings for $j=1,2,\ldots$.
\begin{definition} \label{DEL}
Let $\delta>0$.
 A set $K$ in $\mathbf{R}^N$ is $\delta$-flat if there are $V\subset\mathbf{R}^{N-1}$, a $C^1$ function $\psi\colon\mathbf{R}^{N-1}\to\mathbf{R}$, and a rotation $A\in SO(N)$ such that
\[
	K = \left\{ \left(x,\psi(x) \right)A \bigm| x \in V \right\}
\]
and $\|\nabla\psi\|_\infty\leq \delta$.
\end{definition}
\begin{lemma} \label{CR}
Let $\Sigma$ be a countably $N-1$ rectifiable set.
 For any $\delta>0$, there is a disjoint countable family $\{K_i\}^\infty_{i=1}$ of compact $\delta$-flat sets and $\mathcal{H}^{N-1}$-measure zero $N_0$ such that
\[
	\Sigma = N_0 \cup \left( \bigcup^\infty_{i=1} K_i \right).
\]
\end{lemma}
%
% 原稿 5/8
%
\begin{proof}
By \cite[Lemma 11.1]{Sim}, there is a countable family of  $C^1$ manifolds $\{M_i\}^\infty_{i=1}$ and $N$ with $\mathcal{H}^{N-1}(N)=0$ such that
\[
	\Sigma \subset N \cup \left( \bigcup^\infty_{i=1} M_i \right).
\]
Since $M_i$ is a $C^1$ manifold, it can be written as a countable family of $\delta$-flat sets.
 Thus, we may assume that $M_i$ is $\delta$-flat.
 We define $\{N_i,\Sigma_i\}^\infty_{k=1}$ inductively by
\begin{align*}
	&N_{{1}} := \Sigma \cap M_1, \quad \Sigma_1 := \Sigma \backslash N_1 \\
	&N_{i+1} := \Sigma_i \cap M_{i+1}, \quad \Sigma_{i+1} := \Sigma_i \backslash N_i\ (i=1,{2,\ldots}).
\end{align*}
Here, $N_i$ is $\mathcal{H}^{N-1}$-measurable and $\mathcal{H}^{N-1}(N_i)<\infty$.
 Since $\mathcal{H}^{N-1}$ is Borel regular, for any $\delta$, there exists a compact set $C\subset N_i$ such that $\mathcal{H}^{N-1}(N_i\backslash C)<\delta$.
 Thus, there is a disjoint countable family $\{M_{ij}\}^\infty_{j=1}$ of compact sets, and an $\mathcal{H}^{N-1}$-zero set $N_{i0}$ such that
\[
	N_i = N_{i0} \cup \left( \bigcup^\infty_{j=1} M_{ij} \right)\ (i=1,2,\ldots).
\]
Indeed, we define a sequence of compact sets $\{M_{ij}\}$ inductively by
\begin{align*}
	& M_{i1} \subset N_i,\\
	& M_{i,j+1} \subset N_i \backslash \bigcup^j_{k=1} M_{ik},\ j=1,2,\ldots
\end{align*}
such that $\mathcal{H}^{N-1} \left(N_i\backslash \bigcup^{{j}}_{k=1}M_{i{k}}\right)<1/2^{{j}}$. Then, setting $N_{i0}=N_i\backslash\bigcup^\infty_{j=1} M_{ij}$ yields the desired decomposition of $N_i$.
 Setting % 改行？
\[
	N_0 = (N\cap\Sigma) \cup \left( \bigcup^\infty_{i=1} N_{i0} \right)
\]
and renumbering $\{M_{ij}\}$ as $\{K_i\}$, the desired decomposition is obtained.
\end{proof}
%
% 原稿 6/8
\subsection{Proof of liminf inequality} \label{SINF}
\begin{proof}[Proof of Theorem \ref{INF}]
By Lemma \ref{CR}, for $\delta\in(0,1)$, we decompose $\Sigma$ as
\[
	\Sigma = N_0 \cup \left( \bigcup^\infty_{i=1} K_i \right),
\]
where $\{K_i\}^\infty_{i=1}$ is a disjoint family of compact $\delta$-flat sets and $\mathcal{H}^{N-1}(N_0)=0$.
 We set
\[
	\Sigma_m = \bigcup^m_{i=1} K_i
\]
and take a disjoint family of open sets $\{U^m_i\}^m_{i=1}$ such that $K_i\subset U^m_i$.
 By definition, $K_i$ is of the form
\[
	K_i = \left\{ \left(x,\psi(x)\right)A_i \bigm|
	x \in V_i \right\}
\]
for some $A_i\in SO(N)$, a compact set $V_i\subset\mathbf{R}^{N-1}$ and $\psi_i\in C^1(\mathbf{R}^{N-1})$ with $\|\nabla\psi_i\|_\infty\leq\delta$.
 Since $D$ is dense in $S^{N-1}$, we are able to take $\nu^i\subset D$, which is close to the normal of the hyperplane
\[
	P_i = \left\{ (x,0)A_i \bigm|
	x \in \mathbf{R}^{N-1} \right\}
\]
for $i=1,\ldots,m$.
 We may assume that $\nu_i$ is normal to $P_i$ and $\|\nabla\psi_i\|_\infty\leq 2\delta$ by rotating slightly.
 See Figure \ref{FRK}.
\begin{figure}[htb]
\centering 
\includegraphics[width=6.5cm]{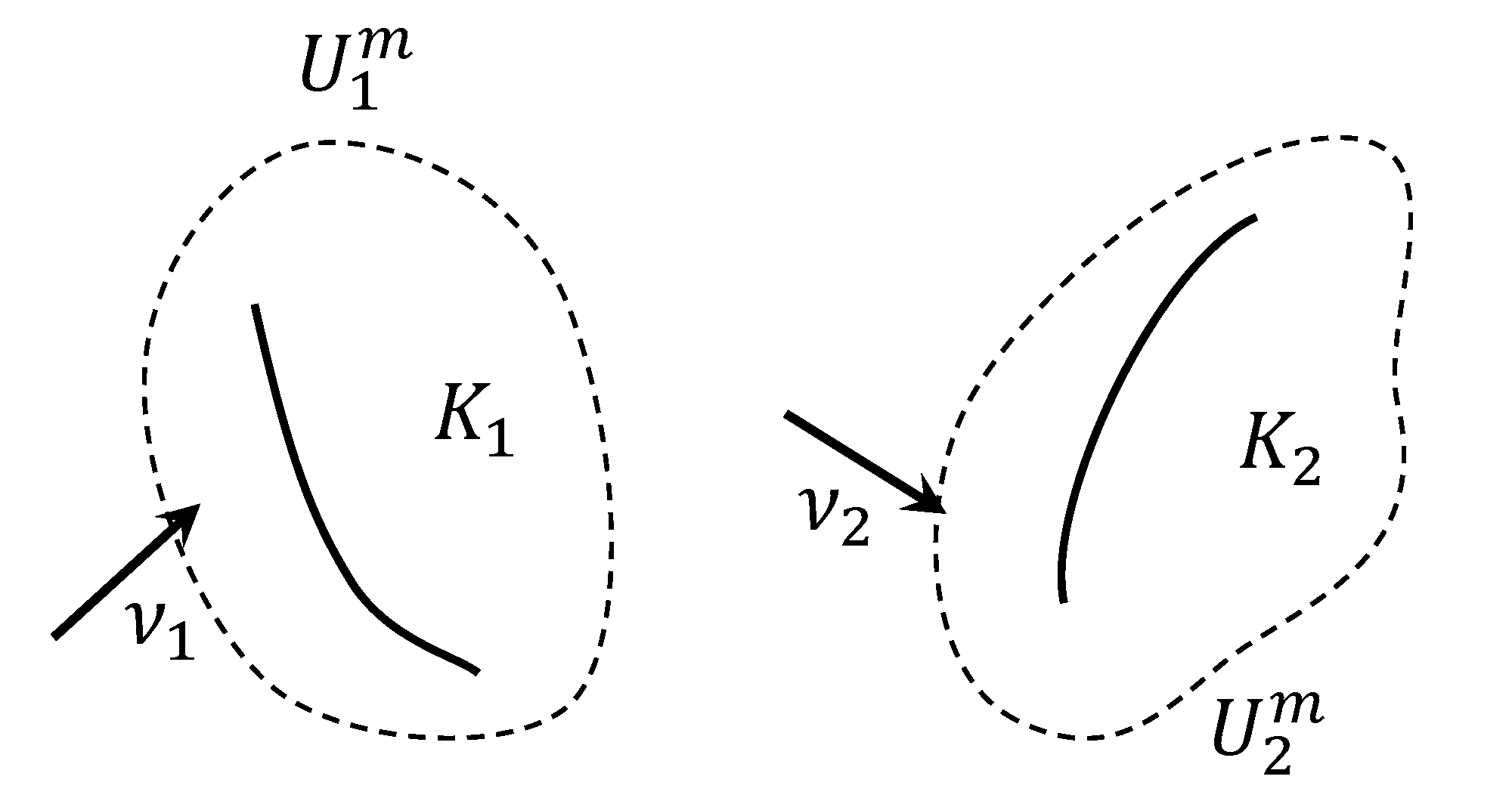}
\caption{The set $\Sigma_2$} \label{FRK}
\end{figure}
We decompose
\[
	E^\varepsilon_\mathrm{sMM}(v_\varepsilon) \geq \sum^m_{i=1} \int_{U^m_i} \left\{ \frac{\varepsilon}{2}|\nabla v_\varepsilon|^2 + \frac{1}{2\varepsilon} F(v_\varepsilon) \right\} \id\mathcal{L}^N.
\]
By slicing, we observe that the right-hand side is
\begin{align*}
	\int_{U^m_i} &{\left\{ \frac{\varepsilon}{2}|\nabla v_\varepsilon|^2 + \frac{1}{2\varepsilon} F(v_\varepsilon) \right\}} \id\mathcal{L}^N\\
	&= \int_{(U^m_i)_{\nu^i}} \left( \int_{(U^m_i)^1_{x,\nu^i}} \left\{ \frac{\varepsilon}{2}|\nabla v_\varepsilon|^2_{x,\nu^i} + \frac{1}{2\varepsilon} F(v_{\varepsilon,x,\nu^i}) \right\} \id t \right) \id\mathcal{L}^{N-1}(x) \\
	&\geq \int_{(U^m_i)_{\nu^i}} \left( \int_{(U^m_i)^1_{x,\nu^i}} \left\{ \frac{\varepsilon}{2}\left|\partial_t(v_{\varepsilon,x,\nu^i}) \right|^2 + \frac{1}{2\varepsilon} F(v_{\varepsilon,x,\nu^i}) \right\} \id t \right) \id\mathcal{L}^{N-1}(x).
\end{align*}
Since $v_\varepsilon\xrightarrow{sg}\Xi$, we see that {$\overline{v_{\varepsilon,x,\nu}}$ converges to $\overline{\Xi_{x,\nu^i}}$ as $\varepsilon \to 0$ in the sense of the graph convergence in a one dimensional setting for $\mathcal{L}^{N-1}$-a.e. $x$.}
% 原稿 7/8
Applying the one-dimensional result \cite[Theorem 2.1 (i)]{GOU}, we have
\begin{equation} \label{ONE}
	\begin{split}
		\liminf_{\varepsilon\to 0} \int_{(U^m_i)^1_{x,\nu^i}} &\left\{ \frac{\varepsilon}{2}\left|\partial_t(v_{\varepsilon,x,\nu^i}) \right|^2 + \frac{1}{2\varepsilon} F(v_{\varepsilon,x,\nu^i}) \right\} \id t\\
		&\geq \sum^\infty_{k=1} 2 \left\{ G \left(\xi^+_{x,\nu^i}(t_k)\right) + G \left(\xi^-_{x,\nu^i}(t_k)\right) \right\}
	\end{split}
\end{equation}
for $\{t_k\}^\infty_{k=1}$, where $\Xi_{x,\nu^i}(t)$ is not a singleton in $(U^m_i)^1_{x,\nu^i}$.
 This set $\{t_k\}^\infty_{k=1}$ contains a unique point $t_x$ such as
\[
	(K_i)^1_{x,\nu^i} \cap (U^m)^1_{x,\nu^i} = \{t_x\},
\]
so the right-hand side of \eqref{ONE} is estimated from below by
\[
	2 \left\{ G \left(\xi^+_{x,\nu^i}(t_x)\right) + G \left(\xi^-_{x,\nu^i}(t_x)\right) \right\}.
\]
Since integration is lower semicontinuous by Fatou's lemma, we now observe that
\[
	\liminf_{\varepsilon\to 0} E^\varepsilon_\mathrm{sMM}(v_\varepsilon) \geq \sum^m_{i=1} \int_{(U^m_i)_{\nu^i}} \widetilde{G} \left(x + t_x \nu^i \right) \id\mathcal{L}^{N-1}(x),
\]
where $\widetilde{G}(x)=2\left\{G \left(\xi^+(x)\right) + G \left(\xi^-(x)\right)\right\}$ ($x\in\Sigma$).
 By the area formula, we see
\begin{align*}
	\int_{K_i} \widetilde{G}(y) {\id\mathcal{H}^{N-1}}(y)
	&= \int_{V_i} \widetilde{G}\left(\left(x,\psi_i(x)\right) A_i \right) \sqrt{1+\left|\nabla\psi_i(x)\right|^2} \id \mathcal{L}^{N-1}(x) \\
	&\leq \sqrt{1+(2\delta)^2} \int_{(U^m_i)_{\nu^i}} \widetilde{G}(x + t_x \nu^i) \id \mathcal{L}^{N-1}(x).
\end{align*}
Thus
\begin{align*}
	\liminf_{\varepsilon\to 0} E^\varepsilon_\mathrm{sMM}(v_\varepsilon)
	&\geq \left( 1+(2\delta)^2 \right)^{-1/2} \sum^m_{i=1} \int_{K_i} \widetilde{G}(x) {\id\mathcal{H}^{N-1}}(x) \\
	&= \left( 1+(2\delta)^2 \right)^{-1/2} \int_{\Sigma_m} \widetilde{G}(x) {\id\mathcal{H}^{N-1}}(x).
\end{align*}
%
% 原稿 8/8
Sending $m\to\infty$ and then $\delta\to 0$, we conclude 
\[
	\liminf_{\varepsilon\to 0} E^\varepsilon_\mathrm{sMM}(v_\varepsilon) \geq \int_\Sigma \widetilde{G}(x) {\id\mathcal{H}^{N-1}}(x).
\]

It remains to prove
\[
	\liminf_{\varepsilon\to 0} \int_J \alpha(v_\varepsilon)j {\id\mathcal{H}^{N-1}}
	\geq \int_{J\cap\Sigma} \left( \min_{\xi^-\leq\xi\leq\xi^+} \alpha(\xi) \right)j {\id\mathcal{H}^{N-1}}
\]
when $v_\varepsilon\xrightarrow{sg}\Xi$.
 It suffices to prove that
\[
	\liminf_{\varepsilon\to 0} \int_{J\cap K_i} \alpha(v_\varepsilon) j{\id\mathcal{H}^{N-1}}
	\geq \int_{J\cap K_i} \left( \min_{\xi^-\leq\xi\leq\xi^+} \alpha(\xi) \right)j {\id\mathcal{H}^{N-1}}.
\]
By slicing, we may reduce the problem in a one-dimensional setting.
 If the dimension equals one, this follows directly from the definition of graph convergence.

The proof is now complete.
\end{proof}
%

%%%%%%%
% 原稿 4-1
\section{Construction of recovery sequences} \label{SCRS} % Section 4

Our goal in this section is to construct what is called a recovery sequence $\{w_\varepsilon\}$ to establish limsup inequality.
\begin{theorem} \label{SUP}
Let $\Omega$ be a bounded domain in $\mathbf{R}^N$.
 Assume that $F$ satisfies (F1) and (F2').
For $\mathcal{J}=(J,j,\alpha)$, assume that $J$ is countably $N-1$ rectifiable in $\Omega$ with a non-negative $\mathcal{H}^{N-1}$-integrable function $j$ on $J$ and that $\alpha\in C(\mathbf{R})$ is non-negative.
 For any $\Xi\in\mathcal{A}_0$ with $E^{0,{\mathcal{J}}}_\mathrm{sMM}(\Xi,\Omega)<\infty$, there exists a sequence $\{w_\varepsilon\}\subset H^1(\Omega)$ such that
\begin{align*}
	& E^{0,{\mathcal{J}}}_\mathrm{sMM}(\Xi,\Omega) \geq \limsup_{\varepsilon\to 0} E^{\varepsilon,{\mathcal{J}}}_\mathrm{sMM}(w_\varepsilon),\\
	& \lim_{\varepsilon\to 0} d_\nu (\Gamma_{w_\varepsilon},\Xi) = 0\quad \text{for all}\quad \nu \in S^{N-1}.
\end{align*}
In particular, $w_\varepsilon\xrightarrow{sg}\Xi$ in $\mathcal{B}_D$ for any $D\subset S^{{N}-1}$ with $\overline{D}=S^{N-1}$.
 By Theorem \ref{INF}, 
\[
	E^{0,{\mathcal{J}}}_\mathrm{sMM}(\Xi,\Omega) = \lim_{\varepsilon\to 0} E^{\varepsilon,{\mathcal{J}}}_\mathrm{sMM}(w_\varepsilon).
\]
\end{theorem}

\subsection{Approximation} \label{SSAP} % Subsection 4.1

We begin with various approximations.
\begin{lemma} \label{APP}
Assume ther same hypotheses concerning $\Omega$ and $S=(J,j,\alpha)$ as in Theorem \ref{SUP}. Assume that $F$ satisfies (F1).
 Assume $\Xi\in\mathcal{A}_0$ so that its singular set $\Sigma=\left\{ y\in\Omega \bigm| \Xi(y)\neq\{1\} \right\}$ is countably $N-1$ rectifiable.
 Let $\delta$ be an arbitrarily fixed positive number.
 Then, there exists a sequence $\{\Xi_m\}^\infty_{m=1}\subset\mathcal{A}_0$ such that the following properties hold:
\begin{enumerate}
\item[(i)] $E^{0,{\mathcal{J}}}_\mathrm{sMM}(\Xi,\Omega) \geq \limsup_{m\to\infty}E^{0,{\mathcal{J}}}_\mathrm{sMM}(\Xi_m,\Omega)$,
\item[(i\hspace{-1pt}i)] $\lim_{m\to\infty}d_\nu(\Xi_m,\Xi)=0$ for all $\nu\in S^{N-1}$,
\item[(i\hspace{-1pt}i\hspace{-1pt}i)] $\Xi_m(y)\subset\Xi(y)$ for all $y\in\Omega$,
\item[(i\hspace{-1pt}v)] the singular set $\Sigma_m=\left\{ y\in\Omega \bigm| \Xi_m(y)\neq\{1\} \right\}$ consists of a disjoint finite union of compact $\delta$-flat sets $\{K_j\}^k_{j=1}$,
\item[{(v)}] $\xi^+_m$, $\xi^-_m$ are constant functions on each $K_j$ ($j=1,\ldots,k$), where $\Xi_m(y)=\left[\xi^-_m(y), \xi^+_m(y)\right]\ni1$ on $\Sigma_m$.
 Here $k$ may depend on $m$.
\end{enumerate}
\end{lemma}

% 原稿 4-2
We recall an elementary fact.
\begin{proposition} \label{SEQ}
Let $h\in C(\mathbf{R})$ be a non-negative function that satisfies $h(1)=0$ and is strictly monotonically increasing for $\sigma\geq 1$.
 Let $\{a_j\}^\infty_{j=1}$ be a sequence such that $a_j\geq 1$ ($j=1,2,\ldots$) and
\[
	\sum^\infty_{j=1} h(a_j) < \infty.
\]
Then
\[
	\lim_{m\to\infty} \sup_{j\geq m} (a_j-1) = 0.
\]
\end{proposition}
\begin{proof}
By monotonicity of $h$ for $\sigma\geq 1$, we observe that
\[
	h \left( \sup_{j\geq m} a_j \right) 
	= \sup_{j\geq m} h(a_j) 
	\leq \sum_{j\geq m} h(a_j) \to 0 
\]
as $m\to\infty$.
 This yields the desired result since $h(\sigma)$ is strictly monotone for $\sigma \geq 1$.
\end{proof}

We next recall a special case of co-area formula \cite[12.7]{Sim} for a countably rectifiable set.
\begin{lemma} \label{CAR}
Let $\Sigma$ be a countably $N-1$ rectifiable set on $\Omega$, and let $g$ be an $\mathcal{H}^{N-1}$-measurable function on $\Sigma$.
 For $\nu\in S^{N-1}$, let $\pi_\nu$ denote the restriction on $\Sigma$ of the orthogonal projection from $\mathbf{R}^N$ to $\Pi_\nu$.
 Then
\[
	\int_\Sigma gJ^*\pi_\nu {\id\mathcal{H}^{N-1}}
	= \int_{\Omega_\nu} \left( \int_{\Sigma^1_{x,\nu}} g_{x,\nu} (t)\id\mathcal{H}^0(t) \right) \id \mathcal{L}^{N-1}(x).
\]
Here $J^*f$ denotes the Jacobian of a mapping $f$ from $\Sigma$ to $\Pi_\nu$.
\end{lemma}
%
% 原稿 4-3
%
\begin{proof}[Proof of {Theorem} \ref{APP}] We divide the proof into two steps.

\noindent
\textit{Step 1.}
 We shall construct $\Xi_m$ satisfying (i)--(i\hspace{-1pt}v).

By Lemma \ref{CR}, we found a disjoint family of compact $\delta$-flat sets $\{K_j\}^\infty_{j=1}$ such that $\Sigma=\bigcup^\infty_{j=1}K_j$ up to $\mathcal{H}^{N-1}$-measure zero set for $\Sigma$ associated with $\Xi$.
 By the co-area formula (Lemma \ref{CAR}) and $J^*\pi_\nu\leq 1$, we observe
\begin{equation} \label{ACA}
	\begin{split}
		\int_{K_j} \widetilde{G}(y) {\id\mathcal{H}^{N-1}}(y)
		&\geq \int_{K_j} \widetilde{G} J^*\pi_\nu {\id\mathcal{H}^{N-1}}\\
		&= \int_{\Omega_\nu} \left(\int_{(K_j)^1_{x,\nu}} \widetilde{G}_{x,\nu}(t)\id\mathcal{H}^0(t) \right) \id \mathcal{L}^{N-1}(x),
	\end{split}
\end{equation}
where $\widetilde{G}(y)=2\bigl(G\left(\xi^+(y)\right)+G\left(\xi^-(y)\right)\bigr)$.
 Since $E^{0,{\mathcal{J}}}_\mathrm{sMM}(\Xi,\Omega)<\infty$, we see that
\begin{equation} \label{BU}
	\sum^\infty_{j=1} \int_{K_j} \widetilde{G} {\id\mathcal{H}^{N-1}}(y) < \infty.
\end{equation}
We then take
\begin{equation*}
\Xi_m(y)= \left \{
\begin{array}{cl}
\left[\xi^-(y), \xi^+(y)\right] &,\ y\in\Sigma_m = \bigcup^m_{j=1} K_j \\
\{1\} &,\ \text{otherwise.}
\end{array}
\right.
\end{equation*}
By definition, (i), (i\hspace{-1pt}i\hspace{-1pt}i), and (i\hspace{-1pt}v) are trivially fulfilled.

It remains to prove (i\hspace{-1pt}i).
 By \eqref{ACA} and \eqref{BU}, we observe that
\[
	\sum^\infty_{j=1}\int_{\Omega_\nu} \left(\int_{(K_j)^1_{x,\nu}} \widetilde{G}_{x,\nu}(t)\id\mathcal{H}^0(t) \right) \id \mathcal{L}^{N-1}(x) < \infty
\]
for $\Xi$.
 Since all integrands are non-negative, the monotone convergence theorem implies that
\[
	\sum^\infty_{j=1}\int_{\Omega_\nu} \left(\int_{(K_j)^1_{x,\nu}} \widetilde{G}_{x,\nu}\id\mathcal{H}^0 \right) \id \mathcal{L}^{N-1}(x)
	= \int_{\Omega_\nu} \left( \sum^\infty_{j=1} \int_{(K_j)^1_{x,\nu}} \widetilde{G}_{x,\nu}\id\mathcal{H}^0 \right) \id \mathcal{L}^{N-1}(x).
\]
% 原稿 4-4
Thus
\[
	 \sum^\infty_{j=1} \int_{(K_j)^1_{x,\nu}}\widetilde{G}_{x,\nu}\id\mathcal{H}^0 < \infty
\]
for $\mathcal{L}^{N-1}$-a.e.\ $x\in\Omega_\nu$.
 Proposition \ref{SEQ} yields % 重複？
\[
	\lim_{m\to 0} \sup_{j\geq m} \sup_{t\in(K_j)^1_{x,\nu}} \left(\xi^+_{x,\nu}(t)-1\right) = 0
\]
and, similarly,
\[
	\lim_{m\to 0} \sup_{j\geq m} \sup_{t\in(K_j)^1_{x,\nu}} \left(1-\xi^-_{x,\nu}(t)\right) = 0.
\]
Since
\[
	d_H\left(\left(\Xi_m\right)_{x,\nu}, \Xi_{x,\nu}\right) 
	= \sup_{j\geq m+1} \sup_{t\in(K_j)^1_{x,\nu}} \max\left\{ \left|\xi^+_{x,\nu}(t)-1\right|, \left|\xi^-_{x,\nu}(t)-1\right| \right\},
\]
we conclude that
\[
	d_H\left(\left(\Xi_m\right)_{x,\nu}, \Xi_{x,\nu}\right) \to 0
\]
as $m\to\infty$ for a.e.\ $x\in\Omega_\nu$.
 Since the integrand of
\[
	d_\nu\left(\Xi_m,\Xi\right) = \int_{\Omega_\nu}
	\frac{d_H\left(\left(\Xi_m\right)_{x,\nu}, \Xi_{x,\nu}\right)}{1+d_H\left(\left(\Xi_m\right)_{x,\nu}, \Xi_{x,\nu}\right)} \id \mathcal{L}^{N-1}(x)
\]
is bounded by $1$, the Lebesgue dominated convergence theorem implies (i\hspace{-1pt}i).

\noindent
\textit{Step 2.}
 We next approximate $\Xi_m$ constructed by Step 1 and construct a sequence $\{\Xi_{m_k}\}^\infty_{k=1}$ satisfying (i)--{(v)} by replacing $\Xi$ with $\Xi_m$.
 If such a sequence exists, a diagonal argument yields the desired sequence.

% 原稿 4-5
We may assume that
\begin{equation*}
\Xi(y)= \left \{
\begin{array}{cl}
\left[\xi^-(y), \xi^+(y)\right], & y\in\Sigma_m = \bigcup^m_{j=1} K_j \\
\{1\}  &,\ \text{otherwise.}
\end{array}
\right.
\end{equation*}
We approximate $\xi^+$ from below.
 For a given integer $n$, we set
\[
	\xi^+_n(y) := \inf \left\{ \xi^+(z) \Bigm| z \in I^k_n \right\}, \quad
	I^k_n := \left\{ y \in\Sigma_m \biggm| \frac{k-1}{n} \leq \xi^+(y)-1 < \frac{k}{n} \right\}
\]
for $k=1,2,\ldots$.
 Since $I^k_n$ is $\mathcal{H}^{N-1}$-measurable set, as in the proof of Lemma \ref{CR}, $I^k_n$ is decomposed as a countable disjoint family of compact sets up to $\mathcal{H}^{N-1}$-measure zero set.
 We approximate $\xi^-$ from above similarly, and we set
\begin{equation*}
\Xi_{m,n}(y)= \left \{
\begin{array}{cl}
\left[\xi^-_n(y), \xi^+_n(y)\right] & ,\ y\in\Sigma_m \\
\{1\}  & ,\ \text{otherwise.}
\end{array}
\right.
\end{equation*}
It is easy to see that $\Xi_{m,n}$ satisfies (i\hspace{-1pt}i\hspace{-1pt}i) and (i\hspace{-1pt}v) by replacing $m$ with $n$.
 Since $E^0_\mathrm{sMM}(\Xi,\Omega)\geq E^0_\mathrm{sMM}(\Xi_{m,n},\Omega)$ and
\[
	\min_{\xi^-_n(y)\leq\xi\leq\xi^+_{{n}}(y)} j(y)\alpha(\xi)
	\to \min_{\xi^-(y)\leq\xi\leq\xi^+(y)} j(y)\alpha(\xi) \ 
	\text{as}\ n \to \infty \ \text{for}\ \mathcal{H}^{N-1}\text{-a.e.}\ y
\]
with bound $j(y)\alpha(1)$, the property (i) follows from the Lebesgue dominated convergence theorem.
 Since
\[
	d_H\left(\left(\Xi_{m,n}\right)_{x,\nu}, \Xi_{x,\nu}\right) 
	= \sup_{t\in(\Sigma_m)^1_{x,\nu}} \max\left\{ \left|\xi^+_{x,\nu}-\xi^+_{n,x,\nu}\right|, \left|\xi^-_{x,\nu}-\xi^-_{n,x,\nu}\right| \right\} \leq 1/n,
\]
we now conclude (i\hspace{-1pt}i) as discussed at the end of Step 1.
\end{proof}

% 原稿 4-6
\subsection{Recovery sequences} \label{SSRC} % Subsection 4.2

In this subsection, we shall prove Theorem \ref{SUP}. % this section？
 An essential step is constructing a recovery sequence $\{w_\varepsilon\}$ when $\Xi$ has a simple structure, and the basic idea is similar to that of \cite{AT,FL}.
 Besides generalization to general $F$ satisfying (F1) and (F2') from $F(z)=(z-1)^2$, our situation is more involved because $\Xi(y)=[0,1]$ for $y\in\Sigma$ in their case, while in our case, $\Xi(y)=\left[\xi^-(y),\xi^+(y)\right]$ for a general $\xi^-\leq1\leq\xi^+$.
 Moreover, we must show the convergence in $d_\nu$ and handle the $\alpha$-term.
\begin{lemma} \label{REC} % Lemma 11
Assume the same hypotheses concerning $\Omega$, $F$, and $\mathcal{J}=(J,j,\alpha)$ as {in} Theorem \ref{SUP}.
 For $\Xi\in\mathcal{A}_0$, assume that its singular set $\Sigma=\left\{x\in\Omega \mid \Xi(x)\neq\{1\} \right\}$ consists of a disjoint finite union of compact $\delta$-flat sets $\{K_j\}^k_{j=1}$, and $\xi^-$ and $\xi^+$ are constant functions in each $K_j$ ($j=1,\ldots,k$), where $\Xi(x)=[\xi^-,\xi^+]$ on $\Sigma$.
 Then there exists a sequence $\{w_\varepsilon\}\subset H^1(\Omega)$ such that
\begin{align*}
	& E^{0,{\mathcal{J}}}_\mathrm{sMM}(\Xi,\Omega) \geq \limsup_{\varepsilon\to 0} E^{\varepsilon,{\mathcal{J}}}_\mathrm{sMM}(w_\varepsilon), \\
	& \lim d_\nu(\Gamma_{w_\varepsilon}, \Xi) = 0 
	\quad\text{for all}\quad \nu \in S^{N-1}.
\end{align*}
\end{lemma}
This lemma follows from the explicit construction of functions $\{w_\varepsilon\}$ similarly to the standard double-well Modica--Mortola functional.
\begin{proof}
We take a disjoint family of open sets $\{U_j\}^k_{j=1}$ with the property $K_j\subset U_j$.
 It suffices to construct a desired sequence $\{w_\varepsilon\}$ so that the support of $w_\varepsilon-1$ is contained in $\bigcup^k_{j=1}U_j$, so we shall construct such $w_\varepsilon$ in each $U_j$.
 We may assume $k=1$ and write $K_1, U_1$ by $K, U$, and $\xi_-,\xi_+$ by $a,b$ ($a\leq1\leq b$) so that
\begin{equation*}
\Xi(y)= \left \{
\begin{array}{cl}
[a, b] & ,\ y\in K, \\
\{1\} & ,\ y \in U \backslash K.
\end{array}
\right.
\end{equation*}
%

% 原稿 4-7
For $c<1$ and $s>0$, let $\psi(s,c)$ be a function determined by
\[
	\int^\psi_c \frac{1}{\sqrt{F(z)}} \id z = s.
\]
By (F1), this equation is uniquely solvable for all $s\in[0,s_*)$ with
\[
	s_* := \int^1_c \frac{1}{\sqrt{F(z)}}\id z.
\]
This $\psi(s,c)$ solves the initial value problem 
\begin{equation} \label{SR}
\left \{
\begin{array}{l}
\displaystyle{\frac{\mathrm{d}\psi}{\mathrm{d}s}} = \sqrt{F(\psi)}, \quad s \in (0,s_*) \\
\psi(0,c)=c,
\end{array}
\right.
\end{equation}
although this ODE may admit many solutions.
 For $c>1$, we parallelly define $\psi$ by % parallelyはOK？
\[
	\int^c_\psi \frac{1}{\sqrt{F(z)}}\id z = s
\]
for $s\in(0,s_*)$ with
\[
	s_* := \int^c_1 \frac{1}{\sqrt{F(z)}} \id z.
\]
In this case, $\psi$ also solves \eqref{SR}.
 We consider the even extension of $\psi$ (still denoted by $\psi$) for $s<0$ so that $\psi(s,c)=\psi(-s,c)$. 
 For the case $c=1$, we set $\psi(s,c)\equiv 1$.
 For $a,b$ with $[a,b]\ni 1$, we consider a rescaled function $\psi_\varepsilon(s,\cdot)=\psi(s/\varepsilon,\cdot)$ and then define
	\begin{equation*}
		\Psi_\varepsilon(s,a,b)= \left \{
		\begin{array}{ll}
			1 & ,\quad  s \leq -2\sqrt{\varepsilon} \\ 
			\alpha_1 s + \beta_1 & , \quad -2\sqrt{\varepsilon} \leq s \leq -\sqrt{\varepsilon}\\
			\psi_\varepsilon(-s,a) & ,\quad -\sqrt{\varepsilon} \leq s \leq 0\\
			\psi_\varepsilon(s,a) & ,\quad 0 \leq s \leq \sqrt{\varepsilon}\\
			\alpha_2 s + \beta_2 & ,\quad \sqrt{\varepsilon} \leq s \leq 2\sqrt{\varepsilon}\\
			\psi_\varepsilon(s-3\sqrt{\varepsilon},b) & ,\quad 2\sqrt{\varepsilon} \leq s \leq 4\sqrt{\varepsilon}\\
			\alpha_3 s + \beta_3 & ,\quad 4\sqrt{\varepsilon} \leq s \leq 5\sqrt{\varepsilon}\\
			1 & ,\quad 5\sqrt{\varepsilon} \leq s
		\end{array}
		\right.
\end{equation*}
with $\alpha_i,\beta_i\in\mathbf{R}$ ($i=1,2,3$)  so that $\Psi_\varepsilon$ is Lipschitz continuous.
 \begin{figure}[thbp]
\centering
\includegraphics[width=0.5\linewidth]{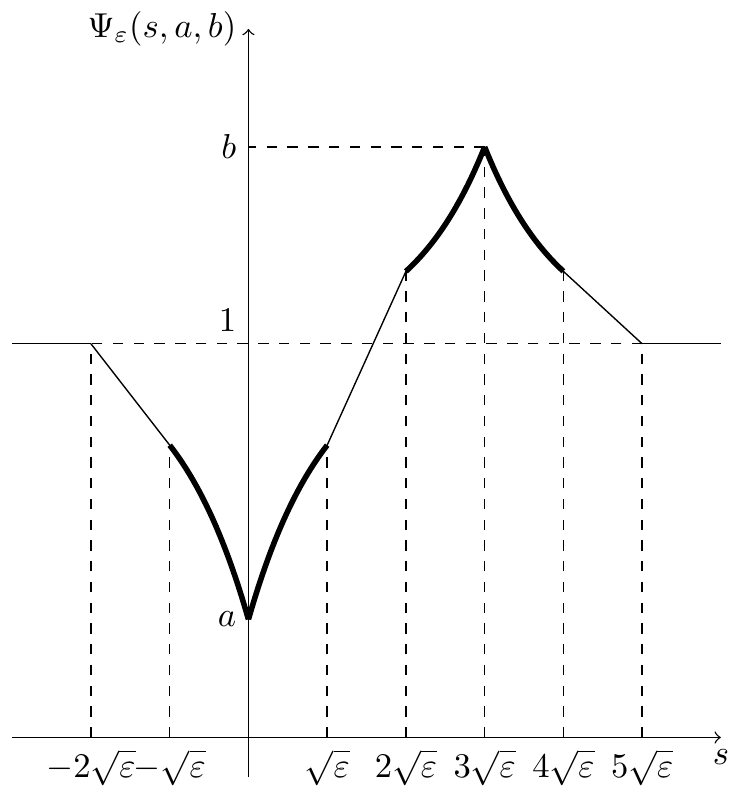}
\label{fig:psi-graph}
\caption{The graph of $\Psi_\varepsilon(s,a,b)$. Thick lines are the part of the graph of $\psi_\varepsilon(s,a)$ or $\psi_\varepsilon(s,b)$, and other parts are linear.}
\end{figure}

% 原稿 4-8
 Let $\eta$ be a minimizer of $\alpha$ in $[a,b]$.
 We first consider the case when $\eta<1$ so that $a\leq\eta<1$.
 In this case, by definition of $\Psi_\varepsilon$, there is a unique $s_0>0$ such that $\Psi_\varepsilon(s_0,a,b)=\eta$.
 We then set
\[
	\varphi_\varepsilon(s,a,b) = \Psi_\varepsilon(s+s_0,a,b).
\]
For the case $\eta\geq1$, we take the smallest positive $s_0>0$ such that $\Psi_\varepsilon(s_0,a,b)=\eta$.
 This $s_0=s_0(\varepsilon)$ is of order $\varepsilon^{3/2}$ as $\varepsilon\to 0$. Since $K$ is a $\delta$-flat surface, it is on the graph of a $C^1$ function $p$. So we can write 
 	\[K = \{(x',p(x'))  \mid x' \in V\}.\]  We set $A := \{(x,z) \mid p(x) \geq z \}$ and $B := \{(x,z) \mid p(x) < z \}.$
 	%We define the positive part of the graph as $A$ and the negative part as $B$.
 	Let $\mathrm{sd}(z)$ be the signed distance of $z$ from $K$, i.e.
 	\[\mathrm{sd}(z) := d({z},A) - d({z},B).\]
 	If $\mathrm{sd}(z)$ is non-negative then we simply write it by $d(z).$  We then take
\[
	w_\varepsilon(z) = \varphi_\varepsilon \left(\mathrm{sd}(z),a,b\right),
\]
%where $d(z)$ is the distance of $z$ from $K$.
 This is the desired sequence such that the support of $w_\varepsilon-1$ is contained in $U$ for sufficiently small $\varepsilon>0$.
 Since $w_\varepsilon$ is Lipschitz continuous, it is clear that  $w_\varepsilon\in H^1(\Omega)$.
 Since
\[
	\nabla w_\varepsilon = (\partial_s \Psi_\varepsilon) \left( \mathrm{sd}(z)+s_0,a,b \right) \nabla \mathrm{sd}(z),
\] 
we have {for $|\mathrm{sd}(z)| <\sqrt{\varepsilon} - s_0$,}
\begin{align*}
	\nabla w_\varepsilon(z) & = (\partial_s \psi_\varepsilon) \left( \mathrm{sd}(z)+s_0,a \right) \nabla \mathrm{sd}(z) \\
	& = \frac{1}{\varepsilon} (\partial_s \psi) \left(\left( \mathrm{sd}(z)+\delta_0\right)/\varepsilon,a \right) \nabla \mathrm{sd}(z).
\end{align*}
Thus, for $z$ with $ -\sqrt{\varepsilon}+s_0 < \mathrm{sd}(z) <\sqrt{\varepsilon} - s_0 $, we see that
\[
	\left| \nabla w_\varepsilon(z) \right|^2
	 = \frac{1}{\varepsilon^2} \bigl| (\partial_s \psi) \left(\left( \mathrm{sd}(z)+s_0\right)/\varepsilon,a \right)\bigr|^2.
\]
% 原稿 4-9
Let $U_\varepsilon$ denote set
\[
U_\varepsilon = \left\{ z \in \Omega \bigm| -\sqrt{\varepsilon} + s_0 < \mathrm{sd}(z) < \sqrt{\varepsilon} - s_0 \right\}.
\]
Since $s_0$ is of order $\varepsilon^{3/2}$, the closure $\overline{U}_\varepsilon$ converges to $K$ in the sense of Hausdorff distance.
 We proceed
\begin{align*}
	E^0_\mathrm{sMM}&(w_\varepsilon,U_\varepsilon) 
	= \int_{U_\varepsilon} \left\{ \frac{\varepsilon}{2} |\nabla w_\varepsilon|^2 + \frac{1}{2\varepsilon} F(w_\varepsilon) \right\} \id\mathcal{L}^N \\
	& = \frac{1}{2\varepsilon} \int_{U_\varepsilon} \bigl| (\partial_s \psi) \left(\left( \mathrm{sd}(z)+s_0\right)/\varepsilon,a \right)\bigr|^2
	+ F \bigl(\psi \left(\left( \mathrm{sd}(z)+s_0\right)/\varepsilon,a \right)\bigr) \id\mathcal{L}^N(z) \\
	& = \frac{1}{\varepsilon} \int_{U_\varepsilon} F \bigl(\psi \left(\left( \mathrm{sd}(z)+s_0\right)/\varepsilon,a \right)\bigr) \id\mathcal{L}^N(z)
\end{align*}
by \eqref{SR}.
 To simplify the notation, we set
\[
	f_\varepsilon(t) = \frac{1}{\varepsilon} F \bigl(\psi \left((t+s_0)/\varepsilon,a \right)\bigr) %\quad\text{for}\quad t>0
\]
and observe that
\[
E^0_\mathrm{sMM} (w_\varepsilon,U_\varepsilon)
	= \int_{U_\varepsilon} f_\varepsilon \left(\mathrm{sd}(z)\right) \id\mathcal{L}^N(z)
	= \int^{\beta(\varepsilon)}_{-\beta(\varepsilon)} f_\varepsilon(t) H(t) \id t,  \quad \beta(\varepsilon) := \sqrt{\varepsilon}-s_0(\varepsilon)
\]
with $H(t):=\mathcal{H}^{N-1}\left(\left\{ z\in U_\varepsilon \bigm| d(z)=t \right\}\right)$ by the co-area formula.
 We set $A(t):=\mathcal{L}^N\left(\left\{ z\in U_\varepsilon \bigm| |\mathrm{sd}(z)| < t \right\}\right)$  and observe that $A(t)=\int^t_{-t} H(s)ds$ by the co-area formula.
 Integrating by parts, we observe that
	\begin{align*}
		\int^{\beta(\varepsilon)}_{-\beta(\varepsilon)} f_\varepsilon(t) H(t) \id t &= \int^{\beta(\varepsilon)}_0 f_\varepsilon(t) H(t) \id t + \int^0_{-\beta(\varepsilon)} f_\varepsilon(t) H(t) \id t\\
		&=\int^{\beta(\varepsilon)}_0 f_\varepsilon(t) \left( H(t) + H(-t) \right) \id t \\
		&= \int^{\beta(\varepsilon)}_0 f_\varepsilon(t) A'(t) \id t \\
		&= f_\varepsilon \left( \beta(\varepsilon) \right) A\left(\beta(\varepsilon) \right) - \int^{\beta(\varepsilon)}_0 f'_\varepsilon(t) A(t) \id t.
\end{align*}
By the relation of Minkowski contents and area \cite[Theorem 3.2.39]{Fe}, we know that
\[
	\lim_{t\downarrow 0} A(t)/2t = \mathcal{H}^{N-1}(K).
\]
% 原稿 4-10
In other words,
\[
	A(t) = 2 \left( \mathcal{H}^{N-1}(K) + \rho(t) \right)t
\]
with $\rho$ such that $\rho(t)\to 0$ as $t\to 0$.
 Thus,
\[
	- \int^{\beta(\varepsilon)}_0 f'_\varepsilon(t) A(t) \id t
	\leq - \int^{\beta(\varepsilon)}_0 f'_\varepsilon(t) 2t\id t \left( \mathcal{H}^{N-1}(K) + \max_{0\leq t\leq\beta(\varepsilon)} \rho(t)_+ \right)
\]
since $f'_\varepsilon(t)\leq 0$.
 Here we invoke (F2') so that $F'(\sigma)\leq0$ for $\sigma<1$.
 We thus observe that
\[
	E^\varepsilon_\mathrm{sMM}(w_\varepsilon,U_\varepsilon)
	\leq f_\varepsilon \left(\beta(\varepsilon)\right) A\left(\beta(\varepsilon)\right)
	- \int^{\beta(\varepsilon)}_0 f'_\varepsilon(t) 2t\id t \left( \mathcal{H}^{N-1}(K) + \max_{0\leq t\leq\beta(\varepsilon)} \rho(t)_+ \right).
\]
Integrating by parts yields
\[
	- \int^{\beta(\varepsilon)}_0 f'_\varepsilon(t) 2t\id t
	= 2 \int^{\beta(\varepsilon)}_0 f_\varepsilon(t) \id t
	-2 f_\varepsilon \left(\beta(\varepsilon)\right) \beta(\varepsilon).
\]
Since $\psi(s)=\psi(s,a)$ solves \eqref{SR}, we see
\begin{align*}
	f_\varepsilon(t-s_0) &= \frac{1}{\varepsilon} F \left(\psi(t/\varepsilon) \right) \\
	&= \frac{1}{\varepsilon} (\partial_s \psi) (t/\varepsilon)\sqrt{F\left(\psi(t/\varepsilon) \right)} \\
	&= -{\frac{\mathrm{d}}{\mathrm{d}t}} \bigl(G\left(\psi(t/\varepsilon) \right)\bigr).
\end{align*}
Thus
\[
	\int^{\beta(\varepsilon)}_0 f_\varepsilon(t) \id t
	= G \left(\psi(s_0/\varepsilon) \right) - G \left(\psi(1/\sqrt{\varepsilon}) \right).
\]
Since $s_0/\varepsilon\to 0$, $\psi(1/\sqrt{\varepsilon},a)\to 1$ as $\varepsilon\to 0$, we obtain
\[
	\lim_{\varepsilon\to 0} \int^{\beta(\varepsilon)}_0 f_\varepsilon(t) dt = G(a).
\]
Combing these manipulations, we obtain that
\begin{align*}
	\limsup_{\varepsilon\to 0}& E^\varepsilon_\mathrm{sMM} (w_\varepsilon,U_\varepsilon) \\
	&\leq \limsup_{\varepsilon\to 0} f_\varepsilon \left(\beta(s)\right) 
	\left\{ A\left(\beta(\varepsilon)\right)
	- 2\left(\mathcal{H}^{N-1} (K)-\max_{0\leq t\leq\beta(\varepsilon)} \left|\rho(t)\right|\right) \beta(\varepsilon) \right\} \\
	&\qquad+ 2\mathcal{H}^{N-1} (K) G(a)
\end{align*}
We thus conclude that
\[
	\limsup_{\varepsilon\to 0} E^\varepsilon_\mathrm{sMM} (w_\varepsilon,U_\varepsilon) \leq 2 \mathcal{H}^{N-1} (K) G(a)
\]
provided that
\[
	\lim_{\varepsilon\to0} f_\varepsilon \left(\beta(\varepsilon)\right) \beta(\varepsilon) < \infty
\]
since $\left(A(t)-2\mathcal{H}^{N-1}(K)t\right)\bigm/t=\rho(t)\to0$ as $t\to0$.
 This condition follows from the following lemma by setting $\varepsilon^{1/2}=\delta$.
 Indeed, we obtain a stronger result
\[
	\limsup_{\varepsilon\to0} f_\varepsilon\left(\beta(\varepsilon)\right) \beta(\varepsilon) \bigm/ \varepsilon^{1/2} < \infty.
\]
\begin{lemma} \label{ELPF}
Assume that $F$ satisfies (F1), (F2').
 Then, for $c\in\mathbf{R}$,
\[
	F \left(\psi(1/\delta,c)\right) \bigm/ \delta^2 \leq (1-c)^2 \
	\quad\text{for}\quad \delta > 0.
\]
\end{lemma}
\begin{proof}[Proof of Lemma \ref{ELPF}]
We may assume $c<1$ since the argument for $c>1$ is symmetric and the case $c=1$ is trivial.
 We write $\psi(s,a)$ by $\psi(s)$.
 By definition and monotonicity (F2') of $F$, we see
\[
	\frac{1}{\delta} = \int^{\psi(1/\delta)}_c \frac{1}{\sqrt{F(z)}} \id z
	\leq \frac{\psi(1/\delta)-c}{\sqrt{F\left(\psi(1/\delta)\right)}}.
\]
Taking the square of both sides, we end up with
\[
	F\left(\psi(1/\delta)\right) \bigm/ \delta^2
	\leq \left(\psi(1/\delta)-c\right)^2 \leq (1-c)^2.
\]
\end{proof}
 %Suppose $z$ satisfies $2\sqrt{\varepsilon} \leq  d(z) + s_0 \leq 4\sqrt{\varepsilon}$.  
 Let $V_{\varepsilon}$ denote the set 
	\begin{align*}
		V_{\varepsilon} := \left\{z \in \Omega \mid 2\sqrt{\varepsilon} < d(z) + s_0 < 4\sqrt{\varepsilon} \right\}.
\end{align*}
{We} observe that
	\begin{align*}
		E^0_\mathrm{sMM} (w_\varepsilon,V_\varepsilon) &= \frac{1}{\varepsilon}  \int_{V_\varepsilon} F \left(\psi \left(\frac{d(z) + s_0 -3\sqrt{\varepsilon}}{\varepsilon}, b \right) \right) \id\mathcal{L}^N(z) \\
		&= \frac{1}{\varepsilon}  \int_{2\sqrt{\varepsilon}-s_0}^{4\sqrt{\varepsilon}-s_0} F \left(\psi \left(\frac{t + s_0 -3\sqrt{\varepsilon}}{\varepsilon}, b \right) \right) H(t) \id t \\
		&= \int_{2\sqrt{\varepsilon}-s_0}^{4\sqrt{\varepsilon}-s_0} \tilde{f}_{\varepsilon}(t - 3\sqrt{\varepsilon}) H(t) \id t,
\end{align*}
 where $\tilde{f}_\varepsilon(t) := \frac{1}{\varepsilon} F \bigl(\psi \left((t+s_0)/\varepsilon,b \right)\bigr) .$ We set $\tilde{A}(t):=\mathcal{L}^N\left(\left\{ z\in V_\varepsilon \bigm| 0\leq d(z) < t \right\}\right)$  and observe that $ \tilde{A}(t)=\int^t_0 H(s)ds$ by the co-area formula. As before, we see
	\[
	\tilde{A}(t) =  \left( \mathcal{H}^{N-1}(K) + \rho(t) \right)t
	\]
	with $\rho$ such that $\rho(t)\to 0$ as $t\to 0$. We set 
	\[b(\varepsilon) :=  \tilde{f}_{\varepsilon}(\sqrt{\varepsilon} - s_0) \tilde{A}(4\sqrt{\varepsilon} - s_0) - \tilde{f}_{\varepsilon}(-\sqrt{\varepsilon} - s_0) \tilde{A}(2\sqrt{\varepsilon} - s_0),\]
	and observe that
	\begin{align*}
		E^0_\mathrm{sMM}(w_\varepsilon,V_\varepsilon) &\leq  b(\varepsilon) - \int_{2\sqrt{\varepsilon} - s_0}^{4\sqrt{\varepsilon} - s_0} \tilde{f}'_{\varepsilon}(t - 3\sqrt{\varepsilon}) t \id t \\
		&\qquad\times\left(\mathcal{H}^{N-1}(K) + \max_{2\sqrt{\varepsilon} - s_0 \leq t \leq 4\sqrt{\varepsilon} - s_0} \rho(t)_{+} \right).
\end{align*}

Integration by parts yields
	\[-\int_{2\sqrt{\varepsilon} - s_0}^{4\sqrt{\varepsilon} - s_0} \tilde{f}'_{\varepsilon}(t - 3\sqrt{\varepsilon}) t \id t = \int_{2\sqrt{\varepsilon} - s_0}^{4\sqrt{\varepsilon} - s_0} \tilde{f}_{\varepsilon} (t - 3\sqrt{\varepsilon}) \id t - 2 \sqrt{\varepsilon} \tilde{f}_{\varepsilon}(\beta(\varepsilon)), \]
and we see
	\[\int_{2\sqrt{\varepsilon} - s_0}^{4\sqrt{\varepsilon} - s_0} \tilde{f}_{\varepsilon} (t - 3\sqrt{\varepsilon}) \id t = 2 \int_{0}^{\beta(\varepsilon)} \tilde{f}_{\varepsilon}(t) \id t. \]
As before, we thus conclude that
\[
\limsup_{\varepsilon\to 0} E^\varepsilon_\mathrm{sMM} (w_\varepsilon,V_\varepsilon) \leq 2 \mathcal{H}^{N-1} (K) G(b).
\]

The part corresponding to $\psi(s,b)$ is similar, and the part where $\Psi_\varepsilon$ is linear will vanish as $\varepsilon\to 0$.
 So, we conclude
\[
	\lim_{\varepsilon\to 0} E^\varepsilon_\mathrm{sMM} (w_\varepsilon,\Omega) \leq E^0_\mathrm{sMM} (\Xi,\Omega).
\]
The term related to $\alpha$ is independent of $\varepsilon$ because of the choice of $s_0$ so that $w_\varepsilon(x)=\eta$ for $x\in K$.

Since $\mathcal{H}^{N-1}(K)<\infty$, by the co-area formula (Lemma \ref{CAR}), $K^1_{x,\nu}$ is a finite set for $\mathcal{L}^{N-1}$-a.e.\ $x\in\Omega_\nu$.
% 修正原稿 6/8
In the Hausdorff sense, $(S_\varepsilon)^1_{x,\nu}\to K^1_{x,\nu}$ holds, as observed in the following lemma for
\[
	S_\varepsilon = \left\{ y\in\mathbf{R}^N \bigm|
	d(y,K) = \varepsilon \right\}.
\]
Therefore, we observe that for $\mathcal{L}^{N-1}$-a.e.\ $x\in\Omega_\nu$,
\[
	\textstyle \limsup^* w_{\varepsilon,x,\nu}=b, \quad
	\liminf_* w_{\varepsilon,x,\nu}=a \ \text{on}\ K^1_{x,\nu}
\]
% 修正原稿 7/8
and outside $K^1_{x,\nu}$, $\limsup^* w_{\varepsilon,x,\nu}=\liminf_* w_{\varepsilon,x,\nu}=1$.
 We conclude that $w_{\varepsilon,x,\nu}$ converges to $\Xi_{x,\nu}$ in the graph sense on $\Omega^1_{x,\nu}$, which proves (i\hspace{-1pt}i).
\end{proof}
\begin{lemma} \label{HAUS}
Let $K$ be a compact set in a bounded open subset $\Omega$ of $\mathbf{R}^N$ and set
\[
	S_\varepsilon = \left\{ y \in \Omega \bigm| d(y,K) = \varepsilon \right\}.
\]
For $\nu\in S^{N-1}$, let $x\in\Omega_\nu$ be such that $K^1_{x,\nu}$ is a non-empty finite set.
 Then, $(S_\varepsilon)^1_{x,\nu}\to K^1_{x,\nu}$ in Hausdorff distance in $\mathbf{R}$ as $\varepsilon\to0$.
\end{lemma}
\begin{proof}[Proof of Lemma \ref{HAUS}]
If $(S_\varepsilon)^1_{x,\nu}$ is not empty, it is clear that
\[
	\sup_{y\in(S_\varepsilon)^1_{x,\nu}}
	d(y, K^1_{x,\nu}) \leq \varepsilon \to 0
\]
as $\varepsilon\to0$.
 It remains to prove that for any $t_0\in K^1_{x,\nu}$, there is a sequence $t_\varepsilon\in(S_\varepsilon)^1_{x,\nu}$ such that $t_\varepsilon\to t_0$ in $\mathbf{R}$.
 We set
\[
	f(\delta) = d \left(x+\nu(t_0 + \delta), K\right)
	\quad\text{for}\quad \delta>0.
\]
Since $t_0$ is isolated and $K$ is compact, we see that $f(\delta)>0$ for sufficiently small $\delta$, say $\delta<\delta_0$.
 Moreover, $f(\delta)$ is continuous on $(0,\delta_0)$ since $K$ is compact.
 Since $f(\delta)\leq\delta$, $f$ satisfies $f(\delta)\to0$ as $\delta\to0$.
 By the intermediate value theorem, for sufficiently small $\varepsilon$, say $\varepsilon\in(0,\varepsilon_0)$, there always exists $\delta(\varepsilon)$ such that $f\left(\delta(\varepsilon)\right)=\varepsilon$, which implies that
\[
	t_\varepsilon = t_0 + \delta(\varepsilon) \in (S_\varepsilon)^1_{x,\nu}.
\]
Since $\delta(\varepsilon)\to0$ as $\varepsilon\to0$, this implies $t_\varepsilon\to t_0$.
The proof is now complete.
\end{proof}
\begin{proof}[Proof of Theorem \ref{SUP}]
This follows from Lemma \ref{APP} and Lemma \ref{REC} by a diagonal argument.
\end{proof}
%

%%%%%%%
% 原稿 5-1
\section{Singular limit of the Kobayashi--Warren--Carter energy} \label{SLKWC} % Section 5

We first recall the Kobayashi--Warren--Carter energy.
 For a given $\alpha\in C(\mathbf{R})$ with $\alpha\geq 0$, we consider the Kobayashi--Warren--Carter energy of the form
\[
	E^\varepsilon_\mathrm{KWC}(u,v)
	= \int_\Omega \alpha(v) |Du| + E^\varepsilon_\mathrm{sMM}(v)
\]
for $u\in BV(\Omega)$ and $v\in H^1(\Omega)$.
 The first term is the weighted total variation of $u$ with weight $w=\alpha(v)$, defined by
\[
	\int_\Omega w|Du|
	:= \sup \left\{ -\int_\Omega u \operatorname{div}\varphi\, \id\mathcal{L}^N \Bigm|
	\left|\varphi(z)\right| \leq w(z)\ \text{a.e.}\ x,\ 
	\varphi\in C^1_c(\Omega) \right\}
\]
for any non-negative Lebesgue measurable function $w$ on $\Omega$.

We next define the functional, which turns out to be a singular limit of the Kobayashi--Warren--Carter energy.
For $\Xi\in\mathcal{A}_0(\Omega)$, let $\Sigma$ be its singular set in the sense that
\[
	\Sigma = \left\{ z\in\Omega \bigm|
	\Xi(z) \neq \{1\} \right\}.
\]
For $u\in BV(\Omega)$, let $J_u$ denote the set of its jump discontinuities.
 In other words, 
\[
	J_u = \left\{ z\in\Omega \backslash \Sigma_0 \bigm|
	j(z) := \left| u(z+0\nu) - u(z-0\nu) \right| > 0 \right\}.
\]
% 原稿 5-2
Here $\nu$ denotes the approximate normal of $J_u$, and $u(z\pm0\nu)$ denotes the trace of $u$ in the direction of $\pm\nu$.
 We consider a triplet $\mathcal{J}(u)={(J_u,j,\alpha)}$ and consider $E^{0,{\mathcal{J}}}_\mathrm{sMM}(\Xi,\Omega)$, whose explicit form is
\[
	E^{0,{\mathcal{J}}}_\mathrm{sMM}(\Xi,\Omega) 
	= E^0_\mathrm{sMM}(\Xi,\Omega)
	+ \int_{{J\cap\Sigma}} j \min_{\xi^-\leq\xi\leq\xi^+} \alpha(\xi)\, {\id\mathcal{H}^{N-1}},
\] 
where $\Xi(z)=\left[\xi^-(z),\xi^+(z)\right]$ for $z\in\Sigma$.
 We then define the limit Kobayashi--Warren--Carter energy:
\[
	E^0_\mathrm{KWC}(u,\Xi,\Omega)
	= \int_{\Omega\backslash J_u} \alpha(1) |Du|
	+ E^{0,\mathcal{J}(u)}_\mathrm{sMM}(\Xi,\Omega),
\]
in which the explicit representation of the second term is
\[
	E^{0,\mathcal{J}(u)}_\mathrm{sMM}(\Xi,\Omega)
	= E^0_\mathrm{sMM}(\Xi,\Omega)
	+ \int_{{J_u\cap\Sigma}} |u^+-u^-|\alpha_0(z)\, {\id\mathcal{H}^{N-1}}(z)
\]
with $u^\pm=u(z\pm0\nu)$ and
\[
	\alpha_0(z) := \min\left\{ \alpha(\xi) \bigm|
	\xi^-(z) \leq \xi \leq \xi^+(z) \right\}.
\]
Here $u^\pm$ are defined by
\begin{align*}
	u^+(x) &:= \inf \left\{ t \in \mathbf{R} \biggm|
	\lim_{r\to0} \frac{\mathcal{L}^{{N}}\left(B_r(x)\cap\{u>t\}\right)}{r^N}=0 \right\}, \\
	u^-(x) &:= \sup \left\{ t \in \mathbf{R} \biggm|
	\lim_{r\to0} \frac{\mathcal{L}^{{N}}\left(B_r(x)\cap\{u<t\}\right)}{r^N}=0 \right\},
\end{align*}
where $B_r(x)$ is the closed ball of radius $r$ centered at $x$ in $\mathbf{R}^N$.
 This is a measure-theoretic upper and lower limit of $u$ at $x$.
 If $u^+(x)=u^-(x)$, we say that $u$ is approximately continuous.
 For more detail, see \cite{Fe}.
 We are now in a position to state our main results rigorously.
\begin{theorem} \label{GKWC1}
Let $\Omega$ be a bounded domain in $\mathbf{R}^N$.
 Assume that $F$ satisfies (F1) and (F2) and that $\alpha\in C(\mathbf{R})$ is non-negative.
\begin{enumerate}
\item[(i)] (liminf inequality) Assume that $\{u_\varepsilon\}_{0<\varepsilon<1}\subset BV(\Omega)$ converges to $u\in BV(\Omega)$ in $L^1$, i.e., $\|u_\varepsilon-u\|_{L^1}\to0$.
 Assume that $\{u_\varepsilon\}_{0<\varepsilon<1}\subset H^1(\Omega)$.
 If $v_\varepsilon\xrightarrow{sg}\Xi$ and $\Xi\in\mathcal{A}_0$, then
\[
	E^0_\mathrm{KMC} (u,\Xi\ \Omega)
	\leq \liminf_{\varepsilon\to0} E^\varepsilon_\mathrm{KMC} (u_\varepsilon,v_\varepsilon).
\]
% 原稿 5-3
\item[(i\hspace{-1pt}i)] (limsup inequality) For any $\Xi\in\mathcal{A}_0$ and $u\in {BV(\Omega)}$, there exists a family of Lipschitz functions $\{w_\varepsilon\}_{0<\varepsilon<1}$ such that
\[
	E^0_\mathrm{KMC} (u,\Xi,\Omega)
	= \lim_{\varepsilon\to0} E^\varepsilon_\mathrm{KMC} (u,w_\varepsilon).
\]
\end{enumerate}
\end{theorem}
\begin{corollary} \label{GKWC2}
Assume the same hypotheses of Theorem \ref{GKWC1}.
 Assume that $f\in L^2(\Omega)$ and $\lambda\geq0$.
 Then the results of Theorem \ref{GKWC1} with $E_{\mathrm{KWC}}^0(u,\Xi,\Omega)$ and $E_{\mathrm{KWC}}^\epsilon(u,\Xi,\Omega)$ being replaced with
 \[
	E^0_\mathrm{KMC} (u,\Xi,\Omega) + \frac{\lambda}{2} \int_\Omega|u-f|^2 \id\mathcal{L}^N
	\quad\text{and}\quad
	E^\varepsilon_\mathrm{KMC} (u,v) + \frac{\lambda}{2} \int_\Omega|u-f|^2 \id\mathcal{L}^N, 
\]
respectively, still hold, provided that $u\in L^2(\Omega)$.
\end{corollary}
\begin{remark} \label{GKWC3}
\begin{enumerate}
\item[(i)] In a one-dimensional case, the liminf inequality here is weaker than \cite[Theorem 2.3 (i)]{GOU} because we assume $u\in BV(\Omega)$, not $u\in BV(\Omega\backslash\Sigma_0)$ with
\[
	\Sigma_0 = \left\{ x \in \Sigma \bigm|
	\alpha_0(z) = 0 \right\}.
\]
It seems possible to extend our results to this situation, but we did not try to avoid technical complications. % complication？
\item[(i\hspace{-1pt}i)] It is clear that Corollary \ref{GKWC2} immediately follows from Theorem \ref{GKWC1} once we admit that $u_\varepsilon\to u$ in $L^1(\Omega)$ implies
\[
	\| u-f \|^2_{L^2} 
	\leq \liminf_{\varepsilon\to0} \| u_\varepsilon-f \|^2_{L^2}.
\]
The last lower semicontinuity holds by Fatou's lemma since $u_{\varepsilon'}\to u$\ $\mathcal{L}^N$-a.e.\ by taking a suitable subsequence.
\end{enumerate}
\end{remark}
%
% 原稿 5-4
\begin{proof}[Proof of Theorem \ref{GKWC1}]
Part (i\hspace{-1pt}i) follows easily from Theorem \ref{SUP}.
 Indeed, taking $w_\varepsilon$ in Theorem \ref{SUP} for $\mathcal{J}=\mathcal{J}(u)$, we see that
\[
	E^{0,{\mathcal{J}}}_\mathrm{sMM} (\Xi,\Omega)
	= \lim_{\varepsilon\to0} E^{\varepsilon,{\mathcal{J}}}_\mathrm{sMM} (w_\varepsilon).
\] 
Since
\[
	\int_\Omega \alpha(w_\varepsilon)|Du|
	= \int_{\Omega\backslash J_u} \alpha(w_\varepsilon)|Du|
	+ \int_{J_u} |u^+ - u^-| \alpha(w_\varepsilon){\id\mathcal{H}^{N-1}},
\]
it suffices to prove that
\[
	\lim_{\varepsilon\to0} \int_{\Omega\backslash J_u} \alpha(w_\varepsilon)|Du|
	= \int_{\Omega\backslash J_u} \alpha(1)|Du|.
\]
Similarly in the proof of Theorem \ref{SUP}, by a diagonal argument, we may assume that $w_\varepsilon$ is bounded.
 Since, by construction, $w_\varepsilon(z)\to1$ for $z\in\Omega\backslash\Sigma$ with a uniform bound for $\alpha(w_\varepsilon)$ and since
\[
	|Du| \left(\Sigma\cap(\Omega\backslash J_u) \right) = 0,
\]
the Lebesgue dominated convergence theorem yields the desired convergence.

It remains to prove (i).
 For this purpose, we recall a few properties of the measure $\langle Du,\nu\rangle$ for $u\in BV(\Omega)$, where $Du$ denotes the distributional gradient of $u$ and $\nu\in S^{N-1}$.
 The following disintegration lemma is found in \cite[Theorem 3.107]{AFP}.
\begin{lemma} \label{5DI}
For $u\in BV(\Omega)$ and $\nu\in S^{N-1}$,
\[
	\left|\langle Du,\nu\rangle\right| = (\mathcal{H}^{N-1} \lfloor \Omega_\nu) \otimes \left|Du_{x,\nu}|\right.
\]
In other words,
\[
	\int_\Omega \varphi \left|\langle Du,\nu\rangle\right|
	= \int_{\Omega_\nu} \left\{ \int_{\Omega^1_{x,\nu}} \varphi_{x,\nu} \left|Du_{x,\nu}\right| \right\}{\id\mathcal{H}^{N-1}}(x)
\]
for any bounded Borel function $\varphi\colon\Omega\to\mathbf{R}$.

\end{lemma}
We also need a representation of the total variation of a vector-valued measure and its component.
 Let $\tau > 0$ and monotone increasing sequence $(a_j)_{j \in \mathbf{Z}}$ such that $a_{j+1} < a_j + \tau$ be given. We consider a division of $\mathbf{R}^N$ into a family of rectangles of the form
\[
	R^\tau_{J,(a_j)} = \prod^N_{i=1}[a_{j_i},a_{j_i+1}), \quad % \Product は存在しない？
	J = (j_1, \ldots, j_N) \in \mathbf{Z}^N 
\]
 We say that the division $\{R^\tau_{J,(a_j)}\}_{J \in \mathbf{Z}^N}$ is a $\tau$-rectangular division associated with $(a_j)$.
 %We may omit to omit $(a_j)$ as explicit and write $\{R^\tau_J\}_{J \in \mathbf{Z}^N}$ in short.
 {Hereafter, we may omit $(a_j)$ and write $\{R_J^\tau\}_{J\in\mathbb{Z}^N}$ in short.}
\begin{lemma} \label{5VT}
Let $\mu$ be an $\mathbf{R}^d$-valued finite Radon measure in a domain $\Omega$ in $\mathbf{R}^N$.
 Let $\{\tau_k\}$ be a decreasing sequence converging to zero as $k\to\infty$.
 Let $\{R^{\tau_k}_J\}_J$ be a fixed $\tau_k$-rectangular division of $\mathbf{R}^N$.
 Let $D$ be a dense subset of $S^{N-1}$.
 Then
\[
	|\mu|(A) = \sup \left\{ \left|\langle \mu,\nu_k \rangle\right|(A) \bigm|
	\nu_k : \Omega\to D\ \text{is constant on}\ {R^{\tau_k}_J \cap \Omega,\ J \in \mathbf{Z}^N,}\ k=1,2,\ldots \right\},
\]
where $A$ is a Borel set.
\end{lemma}
We postpone its proof to the end of this section.

We shall prove (i).
 We recall the decomposition of $\Sigma$ into a countable disjoint union of $\delta$-flat compact sets $K_i$ up to $\mathcal{H}^{N-1}$-measure zero set, and take the corresponding ${\nu^i\in D}$ as in Theorem \ref{INF}.
 We use the notation in Theorem \ref{INF}.
 We may assume that $\bigcap^\infty_{m=1}U^m_i=K_i$.
 By Lemma \ref{5DI}, we proceed
\begin{align*}
	\int_{U^m_i} \alpha(v_\varepsilon)\left|Du_\varepsilon\right|
	&\geq \int_{U^m_i} \alpha(v_\varepsilon)\left|\langle Du_\varepsilon, \nu^i\rangle\right|\\
	&= \int_{(U^m_i)_{\nu^i}} \left\{ \int_{(U^m_i)^1_{x,\nu^i}} \alpha(v_{\varepsilon,x,\nu^i})\left| Du_{\varepsilon,x,\nu^i} \right| \right\}{\id\mathcal{H}^{N-1}}(x).
\end{align*}
By one dimensional result \cite[Lemma 5.1]{GOU}, we see that
\begin{align*}
	&\liminf_{\varepsilon\to0} \int_{(U^m_i)^1_{x,\nu^i}} \alpha(v_{\varepsilon,x,\nu^i})\left|Du_{\varepsilon,x,\nu^i}\right| \\
	&\geq \int_{(U^m_i\backslash\Sigma)^1_{x,\nu^i}} \alpha(1)\left|Du_{x,\nu^i}\right|
	+ \sum_{t\in(\Sigma\cap U^m_i)^1_{x,\nu^i}} \left( \min_{\xi^-_{x,\nu^i}\leq\xi\leq\xi^+_{x,\nu^i}} \alpha(\xi) \right)
	\left|u^+_{x,\nu^i}-u^-_{x,\nu^i}\right|(t).
\end{align*}
(In \cite[Lemma 5.1]{GOU}, $\alpha(v)$ is taken as $v^2$, but the proof works for general $\alpha$.
 In \cite[Lemma 5.1]{GOU}, $|\xi^-_i|^2$ should be $\left((\xi^-_i)_+\right)^2$.)
 The last term is bounded from below by
\[
	\alpha_0 \left(x + t^i_x \nu^i\right) \left|u^+ - u^-\right| \left(x + t^i_x \nu^i\right)
\]
since $(K^m_i)^1_{x,\nu^i}$ (${\subset \left(\Sigma\cap U^m_i \right)^1_{x,\nu^i} }$) is a singleton $\{t^i_x\}$. % カッコは式？
 By the area formula, we see
\begin{align}
 	\int_{K^m_i} &\alpha_0 \left|u^+ - u^-\right|{\id\mathcal{H}^{N-1}}\\
	&\leq \sqrt{1+(2\delta)^2} \int_{(K^m_i)_{\nu^i}} \alpha_0 \left(x + t^i_x \nu^i\right) \left|u^+ - u^-\right|  \left(x + t^i_x \nu^i \right){\id\mathcal{H}^{N-1}} (x).
\end{align}
%文末? ^2の範囲？%
Combining these observations, by Fatou's lemma, we conclude that % カンマなし
\[
	\liminf_{\varepsilon\to0} \int_{U^m_i} \alpha\left(v_\varepsilon\right) \left|Du_\varepsilon\right|
	\geq \frac{1}{\sqrt{1+(2\delta)^{2}}} \int_{K^m_i} \alpha_0 \left|u^+ - u^-\right|{\id\mathcal{H}^{N-1}}. % ^2の範囲？
\]
Adding from $i=1$ to $m$, we conclude that
\[
	\liminf_{\varepsilon\to0} \int_{V^m} \alpha\left(v_\varepsilon\right) \left|Du_\varepsilon\right|
	\geq \frac{1}{\sqrt{1+(2\delta)^{2}}} \int_{\Sigma^m} \alpha_0 \left|u^+ - u^-\right|{\id\mathcal{H}^{N-1}} % ^2の範囲？
\]
for $V^m=\bigcup^m_{i=1}U^m_i$.

For $W^m=\Omega\backslash V^m$, we take $\nu\in D$ and argue in the same way to get
\begin{align*} % 揃える位置？
	\liminf_{\varepsilon\to0}& \int_{W^m} \alpha(v_\varepsilon)\left|Du_\varepsilon\right|\\
	&\geq \int_{(W^m)_\nu} \Biggl\{ \int_{(W^m\backslash\Sigma)^1_{x,\nu}} \alpha(1)\left|Du_{x,\nu}\right| \\
	&\qquad+ \sum_{t\in(\Sigma\cap W^m)^1_{x,\nu}} \left( \min_{\xi^-_{x,\nu}\leq\xi\leq\xi^+_{x,\nu}} \alpha(\xi) \right)
	\left|u^+_{x,\nu}-u^-_{x,\nu}\right|(t) \Biggr\}{\id\mathcal{H}^{N-1}}(x) \\
	&\geq \alpha(1) \int_{(W^m)_\nu} \left\{ \int_{(W^m\backslash\Sigma)^1_{x,\nu}} \left|Du_{x,\nu}\right| \right\}{\id\mathcal{H}^{N-1}}(x) \\
	&= \alpha(1) \int_{W^m\backslash\Sigma} \left|\langle Du, \nu \rangle\right|.
\end{align*}
The last equality follows from Lemma \ref{5DI}.
 Since $W^m\cap\Sigma_m=\emptyset$, combining the estimate of the integral on $V^m$, we now observe that
\begin{align*}
	\liminf_{\varepsilon\to0}& \int_\Omega \alpha(v_\varepsilon)\left|Du_\varepsilon\right|\\
	&\geq \liminf_{\varepsilon\to0} \int_{W^m\backslash(\Sigma\backslash\Sigma_m)} \alpha(v_\varepsilon) \left|\langle Du,\nu \rangle\right| 
	+ \liminf_{\varepsilon\to0} \int_{V^m} \alpha(v_\varepsilon) \left| Du_\varepsilon \right| \\ 
	&\geq \alpha(1) \int_{W^m\backslash(\Sigma\backslash\Sigma_m)} \left|\langle Du,\nu \rangle\right|
	+ \frac{1}{\sqrt{1+(2\delta)^2}} \int_{\Sigma_m} \alpha_0 \left| u^+ - u^- \right|{\id\mathcal{H}^{N-1}}. % ^2の範囲？
\end{align*}
Passing $m$ to $\infty$ yields
\[
	\liminf_{\varepsilon\to0} \int_\Omega \alpha(v_\varepsilon)\left|Du_\varepsilon\right|
	\geq \alpha(1) \int_{\Omega\backslash\Sigma} \left|\langle Du,\nu \rangle\right|
	+ \frac{1}{\sqrt{1+(2\delta)^2}} \int_\Sigma \alpha_0 \left| u^+ - u^- \right|{\id\mathcal{H}^{N-1}}
\]
by Fatou's lemma.
 Since $\delta>0$ can be taken arbitrarily, we now conclude that
\[
	{\liminf_{\varepsilon\to0}} \int_\Omega \alpha(v_\varepsilon)\left|Du_\varepsilon\right|
	\geq \alpha(1) \int_{\Omega\backslash\Sigma} \left|\langle Du,\nu \rangle\right|
	+ \int_{\Omega\cap\Sigma} \alpha_0 \left| u^+ - u^- \right|{\id\mathcal{H}^{N-1}}.
\]
For any $\nu \in D$, we may replace $\Omega$ with an open set in $\Omega$, for example, $\Omega_0 \cap \Omega$ where $\Omega_0$ is an open rectangle.
 Applying the co-area formula (or Fubini's theorem) to the projection $(x_1,\ldots,x_N)\longmapsto x_i$, we have $\mathcal{H}^{N-1}\left(\Sigma\cap\{x_i=q\}\right)=0$ for $\mathcal{L}^1$-a.e.\ $q$, since otherwise, $\mathcal{L}^N(\Sigma)>0$. % \mapsto？　-a.e.のハイフン前後にスペース？
 Thus, for any $\tau>0$, there is a $\tau$-rectangular division $\{R^\tau_J\}_J$ with $\mathcal{H}^{N-1}({\partial R^\tau_J\cap\Sigma})=0$.
 Since $\mathcal{H}^{N-1}({\partial R^\tau_J\cap\Sigma})=0$, by dividing $\Omega$ into ${\{\Omega\cap R^\tau_J\}_J}$, we conclude that
\[
	{\liminf_{\varepsilon\to0}} \int_\Omega \alpha(v_\varepsilon)\left|Du_\varepsilon\right|
	\geq \alpha(1) \int_{\Omega\backslash\Sigma} \left|\left\langle Du,\nu(x) \right\rangle\right|
	+ \int_{\Omega\cap\Sigma} \alpha_0 \left| u^+ - u^- \right|{\id\mathcal{H}^{N-1}}
\]
where $\nu:\Omega\to D$ is a constant on each rectangle.
Applying Lemma \ref{5VT}, we now conclude that
\[
	{\liminf_{\varepsilon\to0}} \int_\Omega \alpha(v_\varepsilon)\left|Du_\varepsilon\right|
	\geq \alpha(1) \int_{\Omega\backslash\Sigma} |Du|
	+ \int_\Omega \alpha_0 \left| u^+ - u^- \right|{\id\mathcal{H}^{N-1}}.
\]
Since we already obtained
\[
	\liminf_{\varepsilon\to0} {E^\varepsilon_\mathrm{sMM}} (v_\varepsilon)
	\geq {E^0_\mathrm{sMM}} (\Xi,\Omega)
\]
by Theorem \ref{INF} and since
\[
	E^\varepsilon_\mathrm{KWC} (v) = {E^\varepsilon_\mathrm{sMM} (v)}
	+ \int_\Omega \alpha(v) |Du|,
\]
the desired liminf inequality follows. % liminf は文章？
\end{proof} % Proof of Theorem \ref{GKWC1} のend？

\begin{proof}[Proof of Lemma \ref{5VT}]
We may assume that $A$ is open since $\mu$ is a Radon measure.
 By duality representation,
\[
	|\mu|(A) = \sup \left\{ \sum^d_{i=1} \int_A \varphi_i \id\mu_i \biggm|
	\varphi = (\varphi_1, \ldots, \varphi_d) \in C_c(A),\ 
	\|\varphi\|_{L^\infty} \leq 1 \right\},
\]
where $C_c(A)$ denotes the space of ($\mathbf{R}^d$-valued) continuous functions compactly supported in $A$ and $\|\varphi\|_\infty:=\sup_{x\in\Omega} \left|\varphi(x)\right|$ with the Euclidean norm $|a|=\langle a,a\rangle^{1/2}$ for $a\in\mathbf{R}^d$.
 Since $\mu(A)<\infty$, by this representation, we see that for any $\delta>0$, there exists $\varphi\in C_c(A)$ with $\|\varphi\|_\infty\leq1$ satisfying
\[
	|\mu|(A) \leq \sum^d_{i=1} \int_A \varphi_i \id\mu_i + \delta.
\]
Since $\varphi$ is uniformly continuous in $A$ and $D$ is dense, for sufficiently large $k$, there is $\tau_k$-rectangular division $\{R^{\tau_k}_J\}$ and $\nu^\delta_k:\Omega\to D$, which is constant on $R^{\tau_k}_J\cap\Omega$ such that
\[
	\left| \varphi - \nu^\delta_k c_k \right| < \delta
	\quad\text{in}\quad R^{\tau_k}_J \cap \Omega
\]
with some constant $0\leq c_k\leq1$. 
This inequality implies that 
\begin{align*}
	\sum^d_{i=1} \int_A \varphi_i \id\mu_i &\leq {\sum_J} {\int_{R^{\tau_k}_J\cap A}} c_k  \langle \mu, \nu^\delta_k \rangle + \delta |\mu|(A) \\
	&\leq \left| \langle\mu,\nu^\delta_k \rangle\right| (A) +\delta |\mu|(A).
\end{align*}
Thus we obtain that
\[
	|\mu|(A) \leq \left| \langle\mu,\nu^\delta_k \rangle\right| (A) + \delta + \delta |\mu|(A).
\]
Hence, by $\mu(A)<\infty$ and the arbitrariness $\delta>0$, we have % 原稿の汚れ？　カンマ？
\begin{multline*}
	|\mu|(A)	\leq \sup \left\{ \left| \langle\mu,\nu_k \rangle\right| (A) \bigm|
	\nu_k : \Omega \to D,\ \nu_k\ \text{is constant on}\right.\\
	\left.\ {R^{\tau_k}_J \cap \Omega,\ J \in \mathbf{Z}^N,}\ k=1,2,\ldots \right\}. % 修正原稿では on の元が in、最後の \}. が重複
\end{multline*}
The reverse inequality is trivial, so the proof is now complete. 
\end{proof}

\section*{Acknowledgments}
The work of the second was supported by the Program for Leading Graduate Schools, MEXT, Japan.
The work of the first author was partly supported by the Japan Society for the Promotion of Science through the grants KAKENHI No.~19H00639, No.~18H05323, No.~17H01091, and by Arithmer Inc.~and Daikin Industries, Ltd.~through collaborative grants. 
The work of the third author was partly supported by the Japan Society for the Promotion of Science through the grants KAKENHI No.~18K13455 {and No.~22K03425}.

\end{document}